\documentclass[a4paper,12pt]{amsart} 

\usepackage{amsmath,amsthm,amssymb} 

\theoremstyle{plain}                
\newtheorem{theorem}{\indent\sc Theorem}[section] 
\newtheorem{lemma}[theorem]{\indent\sc Lemma}
\newtheorem{proposition}[theorem]{\indent\sc Proposition}
\theoremstyle{definition}           
\newtheorem{definition}[theorem]{\indent\sc Definition}
\newtheorem{example}[theorem]{\indent\sc Example}
\theoremstyle{remark}               
\newtheorem{remark}[theorem]{\indent\sc Remark}
\newtheorem{notation}[theorem]{\indent\sc Notation}

\newcommand{\R}{{\mathbb{R}}}                
\newcommand{\Z}{{\mathbb{Z}}}                

\begin{document}
\title[RDEs on Foliated Spaces]
{Stochastic Flows and
  Rough Differential Equations on Foliated Spaces}
\author
{Yuzuru Inahama and Kiyotaka Suzaki}


\begin{abstract}
  Stochastic differential equations (SDEs)
  on compact
  foliated spaces were introduced a few years ago.
  As a corollary,
  a leafwise Brownian motion on a compact
  foliated space was obtained as a solution to an SDE.
  In this paper we construct
  stochastic flows associated with
  the SDEs by using rough path theory,
  which is something like a ``deterministic version"
  of It\^o's SDE theory.
\end{abstract}

\maketitle

\section{Introduction}
Let ${\mathcal M}$ be a manifold
or an Euclidean space
and let  $V_i ~(0 \le i \le d)$ be smooth
vector fields on ${\mathcal M}$.
Consider the following stochastic differential equation
(SDE for short) on ${\mathcal M}$
of Stratonovich type:
\begin{equation}\label{eq.190823-1}
  dx_t=\sum_{i=1}^d V_i (x_t)
  \circ dw^i_{t} + V_0 (x_t) dt,
  \quad
  x_0=m \in {\mathcal M},
\end{equation}
where $(w_t)_{t \ge 0}=(w^1_t, \ldots, w^d_t)_{t \ge 0}$ is
$d$-dimensional Brownian motion.
We assume for simplicity that $V_i~(0 \le i \le d)$
are sufficiently nice
so that a unique solution exists for every $m$
and the explosion of solution never occurs.

An ${\mathcal M}$-valued
random field $(x(t,m))_{t \ge 0, m \in {\mathcal M}}$
is called the stochastic flow associated with
SDE \eqref{eq.190823-1}
if the process $t \mapsto x(t,m)$ solves
SDE \eqref{eq.190823-1}  for every $m$.
It is well-known that, under mild assumptions, such a
stochastic flow exists and $m \mapsto x(t,m)$ is
a diffeomorphism for every $t \ge 0$, a.s.
(See Kunita \cite{Ku} for example.)
Stochastic flows play an indispensable role in SDE theory.

In a sense, the most difficult part of the theory
of stochastic flow is the existence of the flow.
The reason is as follows.
Since SDE  \eqref{eq.190823-1}  is solved for each fixed $m$,
the exceptional subset on which the SDE does not hold
depends on $m$.
However, there are uncountably many $m$'s.
So, it is not obvious a priori if there exists a common
exceptional subset independent of $m$.

A standard (and probably the only)
method in It\^o calculus  to overcome this difficulty is
Kolmogorov's continuity criterion
(see \cite[Section 1.4]{Ku} for example).
Loosely, this criterion can be viewed as
an ``averaged" H\"older continuity of $(t,m) \mapsto x(t,m)$.

In this paper
we are concerned with SDEs on
compact foliated spaces.
First, let us quickly recall foliated spaces.
A foliated manifold is a $n$-dimensional manifold with a decomposition into  connected, injectively immersed submanifolds of the same dimension $p$.
It is also required that the decomposition is locally
of the form $\sqcup\{B_{1}\times \{z\}\,:\, z\in B_{2}\}$, where $B_{1}$ is a connected open subset of $\R^{p}$ and $B_{2}$ is an open subset of $\R^{n-p}$.
The components are called the leaves.
A foliated space is a topological space defined by
allowing the transversal $B_{2}$ to be an open subset of a more general locally compact, second countable, metrizable space
instead of one of Euclidean space $\R^{n-p}$.
In particular, a foliated space embedded in a manifold is called a lamination. Foliated manifolds, laminations, and
consequently foliated spaces are regarded as generalization of dynamical systems. For example, it is known that a locally free action of a connected Lie group on a locally compact, separable, metric space induces a foliated space with the the orbits as the leaves (\cite[Theorem 11.3.9]{CC1}).

In \cite{G}, Garnett introduced Brownian motion along the
leaves on a compact foliated Riemannian manifold to extend
the ergodic theory of dynamical systems to that of foliated manifolds.
Nowadays such a stochastic process and the invariant measures are called a leafwise Brownian motion and harmonic measures.
Candel \cite{C} improved Garnett\rq s approach by using a method
for solving evolution equations and the Hille-Yosida theorem
to construct a Markov semigroup generated by a leafwise elliptic differential operator on a compact foliated space. We call such a process a leafwise diffusion process.
For a recent development of the ergodic theory
of foliated spaces with harmonic measures, see \cite{Ngu} for example.

It was \cite{Suz} that first introduced and studied SDEs on a compact foliated space.
It should be noted that \cite{Ka} discusses about some SDEs
on a special foliated manifold.
It was shown in \cite{Suz} that if $V_i~(0 \le i \le d)$
are leafwise smooth vector fields on a
compact foliated space ${\mathcal M}$,
then SDE \eqref{eq.190823-1} has a unique solution
without explosion for every $m$.
The solution gives a leafwise diffusion process.
In particular, a leafwise
Brownian motion on ${\mathcal M}$
was obtained from a solution to an SDE.
In the case of foliated spaces,
however, the existence of a stochastic flow
associated with SDE \eqref{eq.190823-1} is not known,
mainly because Kolmogorov's criterion is not available
for the following reasons.
The criterion can be used only
if ${\mathcal M}$ looks locally like a Euclidean space,
but a foliated space is in general just a locally compact,
separable metric space.
Even when ${\mathcal M}$ happens to have a manifold structure,
it looks very hard to prove an averaged H\"older continuity
since many objects are assumed to be only continuous
in the transversal direction.
Therefore, proving the existence seems highly non-trivial
(and almost impossible in the framework of It\^o calculus).

To overcome this difficulty, we use rough path theory
instead of standard SDE theory in It\^o calculus.
Rough path theory can be viewed as a deterministic
version of SDE theory.
In particular, rough differential equations (differential equations in the rough path sense)
are generalized controlled ordinary differential equations
and involve no probability measures.
In what follows, we call them RDEs for short.
It is also known that flows associated RDEs
exist and have nice properties.
Therefore, RDE theory is basically a kind of
real analysis and, in particular,
the solution map for an RDE is  continuous
in all input data including the initial value
(Lyons' continuity theorem).
It is of course possible to reprove the existence of
stochastic flow associated SDE \eqref{eq.190823-1}
on a Euclidean space under mild assumptions.

The main objective of this paper is
to prove the existence of the stochastic flow
associated SDE  \eqref{eq.190823-1}
on a compact foliated space by using rough path theory.
Once RDEs are introduced on a compact foliated space,
it turns out that
almost nothing is very difficult from the viewpoint of rough path theory.
We believe that this shows the power of rough path theory.

The following is our main theorem.
Let $\mu$ be the $d$-dimensional Wiener measure on
$C_0 ([0,T], \R^d)$,
the space of continuous paths in $\R^d$
which start at the origin.
We denote by $(w_t)_{0 \le t \le T}= (w_t^1, \ldots, w_t^d)_{0 \le t \le T}$
the canonical realization of $d$-dimensional Brownian motion.
Here, $T \in (0, \infty)$ is an  arbitrary time horizon.
The terminologies
in the statement will be explained in details in later sections.

\begin{theorem}\label{thm.Main}
  Let ${\mathcal M}$ be a compact foliated spaces
  and assume that
  $V_i~(0 \le i \le d)$ are leafwise
  $C^{k+3}$-vector fields on ${\mathcal M}$ $(k \ge 0)$.
  Then, a stochastic flow
  associated with SDE \eqref{eq.190823-1} exists.
  Moreover, the flow is almost surely
  a leaf preserving leafwise
  smooth diffeomorphism
  of ${\mathcal M}$ for each fixed $t \in [0,T]$.

  A more precise statement is as follows.
  There exist a $\mu$-null set ${\mathcal N}$
  and an ${\mathcal M}$-valued random field $\{x (t,m, w) \}$ indexed by
  $(t,m) \in [0,T] \times {\mathcal M}$ defined on the probability space
  $(C_0 ([0,T], \R^d), \mu)$ such that the following hold:
  \begin{itemize}
    \item
          For every $m \in {\mathcal M}$ and
          $w \notin {\mathcal N}$, the mapping $t \mapsto x (t,m, w)$ is
          continuous and solves SDE \eqref{eq.190823-1}
          with initial condition $m$.
    \item
          For every
          $t \in [0,T]$ and $w \notin {\mathcal N}$,
          the mapping $m \mapsto x (t,m, w)$ is a leaf preserving leafwise
          $C^k$-diffeomorphism of ${\mathcal M}$.
    \item
          For every
          $w \notin {\mathcal N}$,
          the mapping $t \mapsto x (t, \bullet, w)$ is continuous from $[0,T]$
          to the space of leafwise
          $C^k$-diffeomorphisms of ${\mathcal M}$.
  \end{itemize}
\end{theorem}

\begin{proof}
  This follows immediately from Propositions \ref{pr.LI}, \ref{pr.homeo}, \ref{pr.diffeo} and \ref{pr.version of SDE}.
  (In fact, a slightly stronger result than this one
  follows from these propositions.)
  An explicit form of ${\mathcal N}$ will be given in \eqref{def.nullset}.
\end{proof}

Although leafwise diffusion processes on foliated spaces were already studied (see {\cite{C, CC2, G} for example)
before \cite{Suz},
it was done from the viewpoint of Markov semigroups and processes,
not of stochastic analysis.
(Here, the term ``stochastic analysis" is used in a
narrower sense than usual.
It means It\^o's SDE theory
and related topics such as rough path theory, Malliavin calculus
and path space analysis.)
Generally speaking, stochastic analysis is very powerful and one could obtain many deep results when it is available.
Indeed, stochastic analysis over Riemannian manifolds is a very rich research topic.
We hope our present work, together with \cite{Suz},
would pave the way for full stochastic analysis over
foliated spaces.

The paper is organized as follows. In Section \ref{sec.foliated spaces} and \ref{sec.rough paths}, we review some fundamental facts for foliated spaces and rough path theory, respectively. Section \ref{sec.RDE on foliated spaces} is devoted to solve RDEs on a compact foliated space (Proposition \ref{pr.gl.sol}) and prove that the solutions define flows of leaf preserving homeomorphisms (Proposition \ref{pr.homeo}) or diffeomorphisms (Proposition \ref{pr.diffeo}).
Finally, we show the existence of stochastic flow of leafwise homeomorphisms associated with SDE \eqref{eq.190823-1} (Proposition \ref{pr.version of SDE}) in Section \ref{sec.applications}. Moreover, our rough path
approach enables us to improve the measurability of the
strong solution constructed in \cite{Suz} of the SDE \eqref{eq.190823-1} (Proposition \ref{pr.version of SDE}). As other applications,
we prove a support theorem (Proposition \ref{pr.support}) and a large deviation principle (Proposition \ref{pr.LD}).

\section{Preliminaries from foliated spaces}\label{sec.foliated spaces}

First of all we introduce some notation and basic facts of foliated spaces.
Let $W_{1}$, $W_{2}$ be topological spaces and $U$ an open subset of $\R^{p}\times W_{1}$.
Let $k$ be a non-negative integer.
A function $f\colon U\to \R$ is said to be of class $C_{L}^{k}$ on $U$ if $f(\cdot,z)$ is of $C^{k}$ for any $z$ and
\[
  U\ni (y,z)\mapsto \partial^{\alpha}_{y}f(y, z)=\frac{\partial^{i_{1}+i_{2}+\cdots +i_{p}}}{\partial^{i_{1}}y^{1}\cdots \partial^{i_{p}}y^{p}} f(y,z)\in \R
\]
is continuous for any multi-index $\alpha=(i_{1}, i_{2},\dots, i_{p})$ with $|\alpha|=i_{1}+i_{2}+\cdots +i_{p}\le k$. A map
$f\colon U\to \R^{d}$ is said to be of class $C_{L}^{k}$ if each of the component functions is of class  $C_{L}^{k}$ on $U$.
Let $V$ be an open subset of $\R^{d}\times W_{2}$.
A map $f\colon U\to V$ is said to be of class $C_{L}^{k}$ if it is locally of the form $f(y, z)=\left(f_{1}(y, z), f_{2}(z)\right)$, where $f_{1}$ is of class $C_{L}^{k}$ and $f_{2}$ is continuous.
In this paper
we sometimes use the term \lq\lq{}leafwise $C^{k}$\rq{}\rq{}
meaning the smoothness defined above. In particular,
we call a map $f\colon U\to V$ a leafwise smooth map if it is of class $C_{L}^{k}$ for any non-negative integer $k$.

Let $\mathcal{M}$, $\mathcal{Z}$ be locally compact, separable, metrizable  spaces. $\mathcal{M}$ is called
a $p$-dimensional foliated space (modelled transversely on $\mathcal{Z}$)
if there exist an open cover $\left\{U_{\alpha}\right\}$ of $\mathcal{M}$ and
homeomorphisms $\left\{\varphi_{\alpha}\colon U_{\alpha}\to B_{\alpha, 1}\times B_{\alpha, 2}\right\}$ such that if $U_{\alpha}\cap U_{\beta}\neq \emptyset$, then $\varphi_{\beta}\circ \varphi_{\alpha}^{-1}\colon \varphi_{\alpha}(U_{\alpha}\cap U_{\beta})\to \varphi_{\beta}(U_{\alpha}\cap U_{\beta})
$
is leafwise smooth,
where $B_{\alpha, 1}$ is
a connected open set in $\R^{p}$ and $B_{\alpha, 2}$ is an open  set in $\mathcal{Z}$.
Such a pair $(U_{\alpha}, \varphi_{\alpha})$ is called a foliated chart and
$\left\{U_{\alpha}\right\}$ of $\mathcal{M}$ is called a foliated atlas.
For convenience we sometimes write $(y_{\alpha}, z_{\alpha})$ instead of $\varphi_{\alpha}$
and regard $U_{\alpha}$ as $B_{\alpha, 1}\times B_{\alpha, 2}$ endowed with the product metric
\[
  \textrm{dist}((y_{\alpha}, z_{\alpha}), (\tilde{y}_{\alpha}, \tilde{z}_{\alpha}))
  =|y_{\alpha}-\tilde{y}_{\alpha}|+\textrm{d}_{\mathcal{Z}}(z_{\alpha}, \tilde{z}_{\alpha}),
\]
where $|\,\cdot\,|$ is the standard Euclidean norm
and $\textrm{d}_{\mathcal{Z}}$ is any metric inducing the topology of $\mathcal{Z}$.
A plaque is a set of the form $\varphi_{\alpha}^{-1}(B_{\alpha, 1}\times \left\{z\right\})$.
We may assume that $\left\{U_{\alpha}\right\}$ is regular, that is:
\begin{itemize}
  \item[($1$)]For each $\alpha$, $\overline{U_{\alpha}}$ is a compact subset of a foliated chart $\left(W_{\alpha}, \psi_{\alpha} \right)$ and $\varphi_{\alpha}=\psi_{\alpha}|_{U_{\alpha}}$. Hence we can consider the plaques of $\overline{U_{\alpha}}$.
  \item[($2$)]$\mathcal{U}$ is locally finite.
  \item[($3$)]Given foliated charts $\left(U_{\alpha}, \varphi_{\alpha}\right)$, $\left(U_{\beta}, \varphi_{\beta}\right)$, and a plaque $P\subset U_{\alpha}$, then $P$ meets at most one plaque of $\overline{U_{\beta}}$.
\end{itemize}
For any $m\in \mathcal{M}$, we set
$$
  \begin{aligned}
    \mathcal{L}_{m}=\left\{x\in \mathcal{M}
    \colon
    \right. & \text{ there exist plaques } P_{1}, P_{2}, \dots, P_{n}                                                                            \\
            & \left.\text{ such that } m\in P_{1}, x\in P_{n} \text{ and } P_{i}\cap P_{i+1}\neq \emptyset \text{ for } 1\le i\le n-1  \right\}.
  \end{aligned}
$$
The subset $\mathcal{L}_{m}$ of $\mathcal{M}$ is called the leaf passing through $m\in \mathcal{M}$.
Let $\{L_{\lambda}\}_{\lambda\in \Lambda}$ be the set of distinct leaves in $\mathcal{M}$. Then
$\mathcal{M}$ is decomposed into the leaves
$\left\{\mathcal{L}_{\lambda}\right\}_{\lambda\in \Lambda}$.
One can easily see that each of the leaves is a $p$-dimensional smooth manifold.
A foliated chart $(U, (y, z))$ containing $m\in \mathcal{M}$  naturally induces a basis
\begin{equation}\label{eq.basis}
  \left\{\left(\frac{\partial}{\partial y^{1}}\right)_{m},
  \left(\frac{\partial}{\partial y^{2}}\right)_{m}, \dots,
  \left(\frac{\partial}{\partial y^{p}}\right)_{m}
  \right\}
\end{equation}
of $T_{m}(\mathcal{L}_{m})$, where $T_{m}(\mathcal{L}_{m})$
is the tangent space of $\mathcal{L}_{m}$ at $m$.
A leafwise $C^{k}$ vector field $V$ on $\mathcal{M}$ is a map
$V\,:\,\mathcal{M}\ni m\mapsto V(m)\in T_{m}(\mathcal{L}_{m})$ whose components
$\left\{\tilde{V}^{i}(y, z)\right\}_{i=1}^{p}$ with respect to the basis $(\ref{eq.basis})$ are leafwise $C^{k}$ in every foliated chart.
References for these fundamentals
and the following examples are found in \cite[Chapter $3$ and $11$]{CC1} and \cite[Chapter $\mathrm{II}$]{MS}.

\begin{example}[the mapping torus (suspension) of a topological dynamical system]
  Let $\mathcal{Z}$ be a compact metric space and
  $F\colon \mathcal{Z}\to \mathcal{Z}$ a homeomorphism. The pair $(\mathcal{Z}, F)$ is called a topological dynamical system.
  We define a $\Z$-action on $\R\times \mathcal{Z}$ by
  $(y, z)\mapsto (y-n, F^{n}(z))$ for $n\in \Z$.
  The quotient space $\mathcal{M}$=$(\R\times \mathcal{Z})/\Z$ is a
  compact one-dimensional foliated space modelled transversely on $\mathcal{Z}$.
\end{example}

\begin{example}[the inverse limit of inverse system of compact manifolds and submersions]
  Let $\{\mathcal{M}_{i}, f_{i}\}_{i\in \Z_{\ge 0}}$ be an
  inverse system, the spaces $\mathcal{M}_{i}$ being compact manifolds of the same dimension $p$ and each of the
  bonding maps $f_{i}\colon \mathcal{M}_{i+1}\to \mathcal{M}_{i}$ being
  a submersion but not a diffeomorphism. Then the
  inverse limit space
  \[
    \mathcal{M}={\mathop{\varprojlim}\limits}_{i\in \Z_{\ge 0}}\mathcal{M}_{i}
    =\{(x_{i})_{i\in \Z_{\ge 0}}\colon f_{i}(x_{i+1})=x_{i} \text{ for any } i\in \Z_{\ge 0}\}
  \]
  is a $p$-dimensional foliated space modelled transversely on a Cantor set.
\end{example}

Next we introduce mapping spaces and jets between foliated spaces in the same way as in the case of manifolds (See \cite[Chapter $2$-$4$]{H} for example).
Let $\mathcal{X}$ be a locally compact, separable, metrizable space and $\mathcal{Y}$ a Polish space, i.e., a separable, completely metrizable space.
We denote by $C(\mathcal{X}, \mathcal{Y})$ the totality of
continuous maps from $\mathcal{X}$ into $\mathcal{Y}$.
The compact-open topology on $C(\mathcal{X}, \mathcal{Y})$ is
generated by the sets of the form
\[
  \{f\in C(\mathcal{X}, \mathcal{Y})\colon f(K)\subset V\}
\]
where $K\subset \mathcal{X}$ is compact and $V\subset \mathcal{Y}$ is open.
For any metric that induces the topology of $\mathcal{Y}$,
the compact-open topology is
the same as that of uniform convergence on every
compact set of $\mathcal{X}$
and hence
$C(\mathcal{X}, \mathcal{Y})$ is a Polish  space.
In particular, if $\mathcal{X}$ is compact and
$\textrm{d}_{\mathcal{Y}}$ is
any (complete)
metric that induces the topology of
$\mathcal{Y}$,
then $C(\mathcal{X}, \mathcal{Y})$ can be metrized
by the (resp. complete) metric
\[
  \textrm{d}_{C(\mathcal{X}, \mathcal{Y})}(f, g)=\sup_{x\in \mathcal{X}}\textrm{d}_{\mathcal{Y}}(f(x), g(x)).
\]
Hence, for any two metrics
$\textrm{d}_{\mathcal{Y}}, \tilde{\textrm{d}}_{\mathcal{Y}}$
that induce the topology of $\mathcal{Y}$,
$\sup_{x\in \mathcal{X}}\textrm{d}_{\mathcal{Y}}(f(x), f_n(x))
  \to 0$ as $n \to \infty$
if and only if
$\sup_{x\in \mathcal{X}} \tilde{\textrm{d}}_{\mathcal{Y}}(f(x), f_n(x))
  \to 0$ as $n \to \infty$.
(The most typical example of $\mathcal{X}$
is a compact interval.
In later sections,
$\mathcal{Y}$ will often be $\overline{U_{\alpha}}$,
whose metric is the product metric induced by $\overline{U_{\alpha}}\simeq \overline{B_{1, \alpha}}\times \overline{B_{2, \alpha}}$ rather than the subspace metric induced by $\overline{U_{\alpha}}\subset \mathcal{M}$}.)

Let $\mathcal{M}$ and $\mathcal{N}$ be foliated spaces modelled
transversely on $\mathcal{Z}_{1}$, $\mathcal{Z}_{2}$ of
dimension $p_{1}$, $p_{2}$ respectively.
It is well-known that locally compact, separable, metrizable spaces are Polish spaces.
A map $f\colon\mathcal{M}\to \mathcal{N}$ is called
foliation preserving
if the image of each leaf in $\mathcal{M}$ is contained in a leaf in $\mathcal{N}$.
In particular,
a map $f\colon\mathcal{M}\to \mathcal{M}$ is called
leaf preserving if $f(m)\in \mathcal{L}_{m}$ for any $m\in \mathcal{M}$.
We denote $C_{L}(\mathcal{M}, \mathcal{N})$ by the totality of
foliation preserving continuous maps from $\mathcal{M}$ into $\mathcal{N}$. It is easy to see that $C_{L}(\mathcal{M}, \mathcal{N})$ is a closed subset of $C(\mathcal{M}, \mathcal{N})$.
Hence $C_{L}(\mathcal{M}, \mathcal{N})$ is also a Polish space.
A pair of foliated charts $(U, \varphi)$ for $\mathcal{M}$ and $(V, \psi)$ for $\mathcal{N}$ is adapted to $f$ if $f(U)\subset V$. Then the map
\[
  \psi\circ f\circ \varphi^{-1}\colon\varphi(U)\to \psi(V)
\]
is called the local representation of $f$ in the given charts, at the point $m$ if $m\in U$.
The local representation is of the form $\psi\circ f\circ \varphi^{-1}(y, z)=(f_{1}(y, z), f_{2}(z))$.
Given a non-negative integer $k$, we call the map $f$ a $C_{L}^{k}$-map if it has local representation
of class $C_{L}^{k}$ at all points in $\mathcal{M}$.
One can easily see that if $f$ is a $C_{L}^{k}$-map then every local representation is of class $C_{L}^{k}$.
We denote by $C_{L}^{k}(\mathcal{M}, \mathcal{N})$ the
totality of $C_{L}^{k}$-maps from $\mathcal{M}$ into $\mathcal{N}$. A homeomorphism $f\colon\mathcal{M}\to \mathcal{N}$ is called a $C_{L}^{k}$-diffeomorphism
if $f\in C_{L}^{k}(\mathcal{M}, \mathcal{N})$ and
$f^{-1}\in C_{L}^{k}(\mathcal{N}, \mathcal{M})$.
Note that $C_{L}^{0}(\mathcal{M}, \mathcal{N})=C_{L}(\mathcal{M}, \mathcal{N})$ and  if $f$ is a $C^{k}_{L}$-diffeomorphism
then the restriction $f|_{\mathcal{L}_{m}}$ to the leaf $\mathcal{L}_{m}$ in $\mathcal{M}$ is a $C^{k}$-diffeomorphism
onto the leaf $\widetilde{\mathcal{L}}_{f(m)}$ in $\mathcal{N}$
for any $m\in \mathcal{M}$.
We denote by ${\rm Diff}_{L}^{k}(\mathcal{M}, \mathcal{N})$
the totality of $C_{L}^{k}$-diffeomorphisms from $\mathcal{M}$ onto $\mathcal{N}$.
We put $C_{L}^{k}(\mathcal{M})=C_{L}^{k}(\mathcal{M}, \R)$, ${\rm Diff}_{L}^{k}(\mathcal{M})={\rm Diff}_{L}^{k}(\mathcal{M}, \mathcal{M})$, and
${\rm Homeo}_{L}(\mathcal{M})={\rm Diff}_{L}^{0}(\mathcal{M})$ for simplicity.

For $k\ge 1$, a $k$-jet from $\mathcal{M}$ into $\mathcal{N}$ is an
equivalence class $[m, f, U]_{k}$, where $U\subset \mathcal{M}$ is an open set, $m\in U$, and $f\colon U\to \mathcal{N}$ is a $C_{L}^{k}$-map. The equivalence relation is defined so that
$[m, f, U]_{k}=[m^{\prime}, f^{\prime}, U^{\prime}]_{k}$ if $m=m^{\prime}$ and
local representations $f=(f_{1}, f_{2})$ and $f^{\prime}=(f^{\prime}_{1}, f^{\prime}_{2})$ in some (and hence any) pair of foliated charts adapted to $f$ and $f^{\prime}$
satisfy $\partial^{\alpha}_{y}f_{1}=\partial^{\alpha}_{y}f^{\prime}_{1}$
and $f_{2}=f^{\prime}_{2}$
at $m$ for any multi-index $\alpha$ with $|\alpha|\le k$.
We denote the $k$-jet of $f$ at $m$ by
\[
  j_{L}^{k}f(m)=[m, f, U]_{k}
\]
and the set of all $k$-jets from $\mathcal{M}$ into $\mathcal{N}$
by $J_{L}^{k}(\mathcal{M}, \mathcal{N})$.

Consider the case when $\mathcal{M}=\R^{p_{1}}\times \mathcal{Z}_{1}$, $\mathcal{N}=\R^{p_{2}}\times \mathcal{Z}_{2}$.
Let $U$ be an open subset of $\R^{p_{1}}\times\mathcal{Z}_{1}$ and $f\in C_{L}^{k}(U, \R^{p_{2}}\times\mathcal{Z}_{2})$.
$f$ is expressed as $f(y, z)=(f_{1}(y, z), f_{2}(z))$ for $(y, z)\in U$. Hence the $k$-jet of $f$ at $(y, z)\in U$ has a
canonical representative
\[
  j_{L}^{k}f(y, z)\simeq(f_{1}(y, z), f_{2}(z), \{\partial_{y}^{\alpha}f_{1}(y, z)\}_{1\le|\alpha|\le k})\in \R^{p_{2}}\times\mathcal{Z}_{2}
  \times \prod_{\ell=1}^{k}\mathcal{L}_{\text{\textrm{sym}}}^{\ell}((\R^{p_{1}})^{\ell}, \R^{p_{2}}),
\]
where $\mathcal{L}_{\text{\textrm{sym}}}^{k}((\R^{p_{1}})^{k}, \R^{p_{2}})$
is the linear space of symmetric $k$-linear maps from
$(\R^{p_{1}})^k$ to $\R^{p_{2}}$.
Conversely, we see that any element of $\R^{p_{2}}\times\mathcal{Z}_{2}
  \times \prod_{\ell=1}^{k}\mathcal{L}_{\text{\textrm{sym}}}^{\ell}((\R^{p_{1}})^{\ell}, \R^{p_{2}})$ comes from a unique $k$-jet at $(y, z)$.
Thus we can identify
\[
  \begin{aligned}
    J_{L}^{k}(\R^{p_{1}}\times\mathcal{Z}_{1}, \R^{p_{2}}\times\mathcal{Z}_{2})
     & \simeq\R^{p_{1}}\times \mathcal{Z}_{1}\times \R^{p_{2}}\times
    \mathcal{Z}_{2}\times \prod_{\ell=1}^{k}\mathcal{L}_{\text{\textrm{sym}}}^{\ell}((\R^{p_{1}})^{\ell}, \R^{p_{2}})                                                                  \\
     & \simeq \R^{p_{1}}\times\R^{p_{2}}\times \prod_{\ell=1}^{k}\mathcal{L}_{\text{\textrm{sym}}}^{\ell}((\R^{p_{1}})^{\ell}, \R^{p_{2}})\times \mathcal{Z}_{1}\times\mathcal{Z}_{2}.
  \end{aligned}
\]
If $U\subset \R^{p_{1}}\times \mathcal{Z}_{1}$ and $V\subset \R^{p_{2}}\times\mathcal{Z}_{2}$ are open sets then
$J_{L}^{k}(U, V)$ is an open subset of $J_{L}^{k}(\R^{p_{1}}\times\mathcal{Z}_{1}, \R^{p_{2}}\times\mathcal{Z}_{2})$.

Let $(U, \varphi)$, $(V, \psi)$ be foliated charts for $\mathcal{M}$, $\mathcal{N}$. The map defined by
\[
  \theta_{U, V}\colon J_{L}^{k}(U, V)\ni j_{L}^{k}f(m)\mapsto j_{L}^{k}(\psi\circ f\circ\varphi^{-1})(\varphi(m))\in J_{L}^{k}(\varphi(U), \psi(V))
\]
is a bijection. Therefore, if we regard $(J_{L}^{k}(U, V), \theta_{U, V})$
as a foliated chart and introduce the topology on $J_{L}^{k}(\mathcal{M}, \mathcal{N})$ by these foliated charts, then
we see that $J_{L}^{k}(\mathcal{M}, \mathcal{N})$ is a foliated space. For each $f\in C_{L}^{k}(\mathcal{M}, \mathcal{N})$, $f$ induces a map $j_{L}^{k}f\colon\mathcal{M}\ni m\mapsto j_{L}^{k}f(m)\in J_{L}^{k}(\mathcal{M}, \mathcal{N})$. One can easily see that
this map is a $C_{L}^{0}$-map.
The next proposition is proved in the same way as in the case of manifolds.
Refer to the original proof
of \cite[Theorem 4.3]{H}.

\begin{proposition}\label{pr.embedding}
  The map
  \[
    j_{L}^{k}\colon C_{L}^{k}(\mathcal{M}, \mathcal{N})\to C_{L}(\mathcal{M}, J_{L}^{k}(\mathcal{M}, \mathcal{N}))
  \]
  is injective and the image is closed.
\end{proposition}

By Proposition \ref{pr.embedding}, if we give
$C_{L}^{k}(\mathcal{M}, \mathcal{N})$ the topology
induced by the map $j_{L}^{k}$, then we see that
$C_{L}^{k}(\mathcal{M}, \mathcal{N})$ is a Polish space.
The following proposition is useful in our arguments.

\begin{proposition}\label{pr.loc convergence}
  Suppose that $\mathcal{M}$, $\mathcal{N}$ are
  foliated spaces
  and suppose further that $\mathcal{M}$ is compact.
  Let $\{f_{n}\}$ be
  a sequence in $C_{L}^{k}(\mathcal{M}, \mathcal{N})$.
  Then, for a map $f\in C_{L}^{k}(\mathcal{M}, \mathcal{N})$, the
  following conditions are equivalent.
  \begin{itemize}
    \item[(i)] $f_{n}$ converges to $f$ as $n\to \infty$ in $C_{L}^{k}(\mathcal{M}, \mathcal{N})$.

    \item[(ii)] For any $m\in \mathcal{M}$,
          there exist a positive integer $N$ and a
          pair of foliated charts $(U, \varphi)$, $(V, \psi)$ adapted to $f$ such that $m\in U$ and for $n\ge N$,
          \begin{itemize}
            \item the pair of foliated charts $(U, \varphi)$, $(V, \psi)$ is adapted to $f_{n}$ and
            \item local representations $f=(f_{1}, f_{2})$ and $f_{n}=(f_{n,1}, f_{n,2})$ satisfy that
                  \begin{equation}\label{eq.loc convergence}
                    \partial^{\alpha}_{y}f_{n,1}\to \partial^{\alpha}_{y} f_{1} \text{ and }
                    f_{n, 2}\to f_{2}
                  \end{equation}
                  uniformly as $n\to \infty$ on
                  a neighborhood of $m$ in $U$
                  for every multi-index $\alpha$ with $|\alpha|\le k$.
          \end{itemize}
  \end{itemize}
\end{proposition}
We omit the proof since it is shown in a similar way
as the previous proposition.

In what follows, we assume that $\mathcal{M}$ is compact.
We set metrics on $\textrm{Homeo}_{L}(\mathcal{M})$ and $\textrm{Diff}_{L}^{k}(\mathcal{M})$, $k \ge 1$, by
\[
  \begin{aligned}
    \textrm{dist}(f, g) & =\textrm{d}_{C(\mathcal{M}, \mathcal{M})}(f, g)
    +\textrm{d}_{C(\mathcal{M}, \mathcal{M})}(f^{-1}, g^{-1}),                                                      \\
    \text{ and }\
    \textrm{dist}(f, g)
                        & =\textrm{d}_{C(\mathcal{M}, J_{L}^{k}(\mathcal{M}, \mathcal{M}))}(j_{L}^{k}f, j_{L}^{k}g)
    +\textrm{d}_{C(\mathcal{M}, J_{L}^{k}(\mathcal{M}, \mathcal{M}))}(j_{L}^{k}(f^{-1}), j_{L}^{k}(g^{-1})),
  \end{aligned}
\]
respectively. Therefore $f_{n}\to f$ as $n\to \infty$ in $\textrm{Homeo}_{L}(\mathcal{M})$ (resp. $\textrm{Diff}_{L}^{k}(\mathcal{M})$)
if and only if
$f_{n}\to f$ and $f_{n}^{-1}\to f^{-1}$ as $n\to \infty$ in $C_{L}(\mathcal{M}, \mathcal{M})$ (resp. $C_{L}^{k}(\mathcal{M}, \mathcal{M})$).

\section{Preliminaries from rough path theory}\label{sec.rough paths}

Now we recall basics of rough path theory very briefly.
Assume $\alpha \in (1/3, 1/2)$ throughout this paper
unless otherwise stated.
Set $\triangle_{[a,b]} =\{ (s,t) \colon a \le s \le t \le b\}$
for $0\le a<b$
and $\triangle_{T} = \triangle_{[0,T]}$ for $T>0$.

First, we introduce $\alpha$-H\"older geometric rough paths.
A continuous map
${\bf w} =(1, {\bf w}^1, {\bf w}^2)$ from $\triangle_{T}$
to the truncated tensor algebra $T^{(2)} (\R^d)= \R \oplus \R^d \oplus
  (\R^d \otimes \R^d)$
is called an $\alpha$-H\"older  rough path
if the following \eqref{def.chen}--\eqref{def.hoel} are satisfied:
\begin{align}\label{def.chen}
  {\bf w}^1_{s,t} = {\bf w}^1_{s,u}+ {\bf w}^1_{u,t},
  \quad
  {\bf w}^2_{s,t} = {\bf w}^2_{s,u}+ {\bf w}^2_{u,t}
  +{\bf w}^1_{s,u}\otimes {\bf w}^1_{u,t}
  \quad
  (0\le s \le u \le t \le T),
  \\
  \|{\bf w}^i \|_{i\alpha}
  :=\sup_{0 \le s<t \le T} \frac{|{\bf w}^i_{s,t} |}{(t-s)^{i \alpha}}
  < \infty \qquad (i=1,2).
  \label{def.hoel}
\end{align}
The set of $\alpha$-H\"older  rough paths
is denoted by $\Omega_{\alpha}([0, T], \R^{d})$,
which is equipped with a natural distance
${\rm d}_{\alpha}({\bf w},  \hat{\bf w})
  := \max_{i=1,2}\|{\bf w}^i -  \hat{\bf w}^i\|_{i\alpha}$.
We will usually write ${\bf w} =({\bf w}^1, {\bf w}^2)$ by omitting
the trivial component ``$1$".

Let $C^{1-\textrm{Hld}}_{0}([0, T], \R^{d})$ be the totality of Lipschitz continuous paths $w=(w_{t})_{0\le t\le T}$ in $\R^{d}$ with $w_{0}=0$. For  $w\in C^{1-\textrm{Hld}}_{0}([0, T], \R^{d})$, we set ${\bf w}^1_{s,t} =w_t -w_s$
and ${\bf w}^2_{s,t} = \int_s^t (w_u -w_s) \otimes dw_u$
by using Riemann-Stieltjes integral.
Then, it is easy to see that ${\bf w} \in \Omega_{\alpha}([0, T], \R^{d})$.
We denote the natural lift map by
${\bf L}\colon C^{1-{\rm Hld}}_0 ([0,T], \R^d)
  \to
  \Omega_{\alpha}([0, T], \R^{d})$, i.e., ${\bf w}={\bf L}(w)$.
The
$\alpha$-H\"older geometric rough path space
$G\Omega_{\alpha}([0, T], \R^{d})$
is defined to be the ${\rm d}_{\alpha}$-closure of
${\bf L} ( C^{1-{\rm Hld}}_0 ([0,T], \R^d))$ in
$\Omega_{\alpha}([0, T], \R^{d})$.
By the way it is defined,
$G\Omega_{\alpha}([0, T], \R^{d})$ is a complete and
separable metric space.
It is well-known that
every
${\bf w}\in G\Omega_{\alpha}([0, T], \R^{d})$ satisfies the following
shuffle relation:
\begin{equation}\label{eq.shfl}
  {\bf w}^{1,j}_{s,t}{\bf w}^{1,k}_{s,t}
  =
  {\bf w}^{2,jk}_{s,t} + {\bf w}^{2,kj}_{s,t}
  \qquad
  ((s,t)\in \triangle_T,
  1\le j,k \le d).
\end{equation}

We consider the following rough differential equation
(with drift)
driven by an $\alpha$-H\"older geometric rough path
${\bf w}=({\bf w}^{1}, {\bf w}^{2})\in G\Omega_{\alpha}([0, T], \R^{d})$:
\begin{equation}\label{RDE1}
  dx_{t}=\sum_{i=1}^d V_i (x_{t})dw^i_{t} + V_0 (x_{t}) dt.
\end{equation}
Here, $V_i~(0 \le i \le d)$ is a vector field on $\R^p$.
When it is regarded as a $\R^{p}$-valued function,
it is written as $\tilde V_i$ or $V_i {\rm Id}$.
As usual,  we set $\|V_i \|_{C_b^k} :=
  \|\tilde V_i \|_{C_b^k}
  = \sum_{l=0}^k \| \nabla^l \tilde V_i\|_\infty$,
where $\| \,\cdot\, \|_\infty$ stands for the sup-norm
($0 \le k <\infty$).
If the vector fields
are $C^3$, then for any given initial condition $\xi\in \R^p$
there exists a unique solution up to an explosion time.
If ${\bf w} =L(w)$ for some $w \in C^{1-{\rm Hld}}_0 ([0,T], \R^d)$,
then the unique solution $(x_t)$ coincides with the one
in the Riemann-Stieltjes sense.
When there is a unique global solution $(x_t)$ with $x_0 =\xi$
for
given ${\bf w}$, $\xi$,
${\bf V}:=[\tilde{V}_1, \ldots, \tilde{V}_d; \tilde{V}_0]$,
we write $x_{t} = \Psi (\xi, {\bf w}, {\bf V})_t$.
If the vector fields are $C^3_b$, then no explosion occurs
(i.e., a unique global solution exists).

In this paper we only consider the first level path of
a solution to an RDE and simply call it a solution.
(The first level path is the component that plays
the role of a path in the usual sense.)
Therefore, a solution $(x_t)$ to RDE \eqref{RDE1}
is an $\alpha$-H\"older path that starts at a certain
initial point.

\begin{remark}\label{rm.ryuha}
  There are in fact several formulations of an RDE.
  In any of them,
  the first level path of a solution coincide.
  The three major formulations are as follows:
  \begin{itemize}
    \item
          Lyons' original formulation:
          A solution of an RDE is also a rough path.
          (See \cite{LQ, LCL} for example.)
    \item
          Gubinelli's controlled path theory:
          A solution of an RDE driven by a rough path is
          not a rough path, but a controlled path with respect to
          the given rough path.
          (See \cite{FH} for example.)
    \item
          Davie's formulation:
          A solution of an RDE is a usual path
          that satisfies a Euler-type short time approximation.
          (See \cite{FV} for example.)
          This formulation has variants e.g. \cite{Bai},
          which will be used in the present paper.
  \end{itemize}
\end{remark}

We now recall Lyons' continuity theorem,
which states that Lyons-It\^o map is (locally Lipschitz)
continuous in all of its arguments.
This is the most important theorem in rough path theory.
In this paper it is quite important that
the solution depends continuously
on the coefficient of the RDE.
\begin{proposition}\label{pr.cont_LI}
  We have the following results for RDE  \eqref{RDE1}.
  \begin{enumerate}
    \item
          For every $\xi \in \R^p$,
          ${\bf w} \in G\Omega_{\alpha}([0, T], \R^{d})$,
          and ${\bf V} \in C^3_b (\R^{p}, \R^{p} \otimes \R^{d+1})$,
          then there exists a unique solution $x=(x_t)_{0\le t \le T}$ to RDE  \eqref{RDE1} with $x_0 =\xi$.
    \item
          Let $C^{\alpha-{\rm Hld}} ([0,T], \R^p)$ be the space of $\alpha$-H\"older continous paths on $[0, T]$ with values in $\R^{p}$. Then, the Lyons-It\^o map
          $(\xi, {\bf w},  {\bf V})
            \mapsto x = \Psi (\xi, {\bf w}, {\bf V})$
          is a locally Lipschitz continuous map from
          $\R^d \times G\Omega_{\alpha}([0, T], \R^{d})
            \times C^3_b (\R^{p}, \R^{p} \otimes \R^{d+1})$ to
          $C^{\alpha-{\rm Hld}} ([0,T], \R^p)$.
          More precisely, we have the following estimate:
          If
          \[
            \max_{i=1,2} \|{\bf w}^i \|_{i\alpha}^{1/i} \le K, \quad
            \max_{i=1,2} \|\hat {\bf w}^i \|_{i\alpha}^{1/i} \le K,\quad
            \| {\bf V}\|_{C^3_b} \le K, \quad
            \| \hat{\bf V}\|_{C^3_b} \le K
          \]
          for $K>0$, then there exists a positive constant
          $C_K$ which depends only on $K$ such that
          \[
            \|  \Psi (\xi, {\bf w}, {\bf V})
            -  \Psi (\hat\xi, \hat{\bf w}, \hat{\bf V})
            \|_{\alpha} \le C_K (|\xi -\hat\xi| +
            {\rm d}_{\alpha}({\bf w},  \hat{\bf w})  +\|{\bf V} -\hat{\bf V}\|_{C^3_b}).
          \]
          Here, the norm on the left hand side stands for the
          $\alpha$-H\"older seminorm.
  \end{enumerate}
\end{proposition}

\begin{proof}
  This is a special case of \cite[Theorem 10.26]{FV}.
\end{proof}

\begin{remark}
  If we consider RDE \eqref{RDE1} on a subinterval
  $[a,b] \subset [0,T]$, then a unique global
  solution $x$ with $x_a =\xi$
  is denoted by
  $x = \Psi_{[a,b]} (\xi, {\bf w}, {\bf V})$ (if it exists).
  Of course, Proposition \ref{pr.cont_LI} above applies to $ \Psi_{[a,b]}$
  with trivial modifications
  and  the local Lipschitz constant $C_K$ can be chosen
  independent of $[a,b]$.
\end{remark}

\begin{proposition}\label{pr.18r^d}
  Let ${\bf w}\in
    G\Omega_{\alpha}([0, T], \R^{d})$ and
  $[a,b]$ be a subinterval of $[0,T]$.
  Assume that $V_i~(0 \le i \le d)$ is $C^3$.
  Then, for $x\in C([a, b], \R^{p})$,
  the following conditions are equivalent.
  \begin{itemize}
    \item[(A)] $x$ is
          the first level path of a unique solution to
          RDE \eqref{RDE1} on $[a,b]$.
    \item[(B)] $x$ satisfies
          \begin{align}
            f(x_{t})-f(x_{s})
             & =\sum_{i=1}^{d}V_{i}f(x_{s}){\bf w}^{1, i}_{s, t}
            +\sum_{j,k=1}^{d}V_{j}V_{k}f(x_{s}){\bf w}^{2, jk}_{s, t}
            \nonumber                                            \\
             & \qquad \qquad
            +
            V_{0}f(x_{s}) (t-s)
            +O(|t-s|^{3\alpha})
            \label{eq:190125-1}
          \end{align}
          for every $f\in C^{3}(\R^{p}, \R)$ and $(s,t)\in \triangle_{[a, b]}$.
  \end{itemize}
\end{proposition}

\begin{remark}\label{rem.190920}
  The precise meaning of \eqref{eq:190125-1}
  is as follows:
  There exists a positive constant $C$ such that
  \begin{align}
    \Bigl|
    f(x_{t})-f(x_{s})
    -\Bigl\{ \sum_{i=1}^{d}V_{i}f(x_{s}){\bf w}^{1, i}_{s, t}
     & +\sum_{j,k=1}^{d}V_{j}V_{k}f(x_{s}){\bf w}^{2, jk}_{s, t}
    \nonumber                                                    \\
     &
    +V_{0}f(x_{s})(t-s)
    \Bigr\}  \Bigr|
    \le
    C	|t-s|^{3\alpha},
    \qquad (s, t)  \in \triangle_{[a,b]}.
    \nonumber
  \end{align}
  Here, $C$ may depend on $f, [a,b], {\bf w}, x$ and $V_i$'s,
  but not on $(s,t)$.
  It should be noted
  that when we write $O(|t-s|^{3\alpha})$
  we do not assume that $t-s$ is small enough.
  Throughout this paper, we will use Landau's ``$O$-symbol"
  in this way.
\end{remark}

\begin{proof}[Proof of Proposition \ref{pr.18r^d}]
  As usual we set
  $\sigma =[\tilde{V}_1, \ldots,  \tilde{V}_d]$ and $b = \tilde{V}_0$,
  which take values in $\R^p \otimes \R^d$ and $\R^p$, respectively.
  Then, the matrix notation for RDE \eqref{RDE1} reads
  \[
    dx_{t}=\sigma(x_{t})d{\bf w}_{t} + b (x_{t})dt.
  \]
  In this proof
  this is understood in
  the sense of controlled rough path theory.

  First, we show (A) implies (B).
  Let $(x, x^{\dagger})$ be a unique solution of the above RDE,
  that is,
  it is a controlled path with respect to ${\bf w}$ and satisfies
  \begin{equation}\label{eq.1019-1}
    x_t -x_0 = \int_0^t \sigma (x_u) d {\bf w}_u + \int_0^t  b(x_u) du,
    \qquad
    x^{\dagger}_t = \sigma (x_t).
  \end{equation}
  Precisely, the integrand in the integration above
  is
  $(\sigma (x), \sigma (x)^{\dagger})
    = (\sigma (x), \nabla\sigma (x)\cdot x^{\dagger})$.
  By a basic estimate for rough path integrals
  (Theorem 4.10 \cite{FH}), we see that for $a \le s \le t\le b$
  \begin{align}
    x_t -x_s
     & =
    \sigma (x_s) {\bf w}^1_{s,t}+ \sigma (x)_s^{\dagger} {\bf w}^2_{s,t}
    + b (x_s) (t-s)
    + O (|t-s|^{3\alpha})
    \nonumber \\
     & =
    \sigma (x_s) {\bf w}^1_{s,t}
    +  \nabla\sigma (x_s)\cdot \sigma (x_s)  \langle {\bf w}^2_{s,t} \rangle
    + b (x_s) (t-s)
    + O (|t-s|^{3\alpha})
    \nonumber \\
     & =
    \sum_{i} V_i {\rm Id} (x_s) {\bf w}^{1,i}_{s,t}
    +
    \sum_{j,k} V_j V_k {\rm Id} (x_s) {\bf w}^{2,jk}_{s,t}
    + V_0 {\rm Id} (x_s)(t-s)
    + O (|t-s|^{3\alpha}).
    \label{eq.1019-2}
  \end{align}
  Here, we set
  $ \nabla\sigma (x)\cdot \sigma (x)  \langle \xi\otimes \eta \rangle
    =  \nabla\sigma (x)  \langle  \sigma (x) \xi, \eta \rangle$
  for $\xi, \eta \in {\mathbb R}^d$
  and
  ${\rm Id}$ stands for the identity map of ${\mathbb R}^p$.

  Plugging this into the Taylor expansion of $f \in C^3$,
  we obtain
  \begin{align}
    f(x_t) - f(x_s)
     & =
    \nabla f(x_s) \langle x_t-x_s \rangle
    +\frac12  \nabla^2 f(x_s) \langle x_t-x_s,  x_t-x_s\rangle
    + O (|x_t-x_s |^{3})
    \nonumber  \\
     & =
    \nabla f(x_s)
    \bigl\langle \sum_{i} V_i {\rm Id} (x_s) {\bf w}^{1,i}_{s,t}
    +
    \sum_{j,k} V_j V_k {\rm Id} (x_s) {\bf w}^{2,jk}_{s,t}
    + V_0 {\rm Id} (x_s)(t-s)
    \bigr\rangle
    \nonumber  \\
     & \qquad+
    \frac12 \nabla^2 f(x_s)
    \bigl\langle
    \sum_{i} V_i {\rm Id} (x_s) {\bf w}^{1,i}_{s,t},
    \sum_{i'} V_{i'} {\rm Id} (x_s) {\bf w}^{1,i'}_{s,t}
    \bigr\rangle
    + O (|t-s|^{3\alpha})
    \nonumber  \\
     & =
    \sum_{i}V_{i}f(x_{s}){\bf w}^{1, i}_{s, t}
    +\sum_{j, k}V_{j}V_{k}f(x_{s}){\bf w}^{2, jk}_{s, t}
    + V_0 f (x_s)(t-s)
    +O(|t-s|^{3\alpha}).
    \label{eq.1019-3}
  \end{align}
  Here, we used the shuffle relation  \eqref{eq.shfl}
  and  the symmetry of $\nabla^2 f$.
  Thus, we have obtained (B) from (A).

  Next, we show (B) implies (A).
  Since the image of $t \mapsto x_t$ is compact,
  there exists an ${\mathbb R}^p$-valued $C^3_b$-function $f$
  which coincides with ${\rm Id}$ on a ball that contains the image.
  Applying (B) to each component of $f$, we get
  \begin{align}
    x_t -x_s
     & =
    \sum_{i} V_i {\rm Id} (x_s) {\bf w}^{1,i}_{s,t}
    +
    \sum_{j,k} V_j V_k {\rm Id} (x_s) {\bf w}^{2,jk}_{s,t}
    + V_0 {\rm Id} (x_s)(t-s)
    + O (|t-s|^{3\alpha})
    \nonumber \\
     & =
    \sigma (x_s) {\bf w}^1_{s,t}
    +  \nabla\sigma (x_s)\cdot \sigma (x_s)  \langle {\bf w}^2_{s,t} \rangle
    + b (x_s)(t-s)
    + O (|t-s|^{3\alpha}).
    \nonumber
  \end{align}
  This implies that, if we set $x^{\dagger} = \sigma (x)$,
  then $(x, x^{\dagger})$ is a controlled path with respect to ${\bf w}$.
  By a composition formula for controlled path
  (Lemma 7.3 \cite{FH}),
  so is
  $(\sigma (x), \sigma (x)^{\dagger})$
  if we set  $\sigma (x)^{\dagger} = \nabla\sigma (x)\cdot x^{\dagger}
    =  \nabla\sigma (x)\cdot \sigma (x)$.
  Hence,  for $a \le s \le t\le b$, we have
  \begin{align}
    x_t -x_s
     & =
    \sigma (x_s) {\bf w}^1_{s,t}
    + \sigma (x)^{\dagger}_s  {\bf w}^2_{s,t}
    + b (x_s)(t-s)
    + O (|t-s|^{3\alpha}).
    \nonumber
  \end{align}

  Fix $s<t$ and let ${\mathcal P}= \{ s=t_0 < t_1 <\cdots < t_N =t \}$
  be a partition of $[s,t] \subset [a,b]$.
  It is obvious that $\sum_{i=1}^N O (|t_i- t_{i-1}|^{3\alpha})
    \to 0$ as the mesh $|{\mathcal P}|$ tends to zero.
  By the Riemann sum approximation of rough
  path integral (Theorem 4.10 \cite{FH}), we see that
  \begin{align*}
    x_t -x_s
     & =
    \sum_{i=1}^N (x_{t_i} -x_{t_{i-1}} )
    \nonumber \\
     & =
    \sum_{i=1}^N
    \Bigl(
    \sigma (x_{t_{i-1}}) {\bf w}^1_{t_{i-1}, t_{i} }
    + \sigma (x)^{\dagger}_{t_{i-1}}  {\bf w}^2_{t_{i-1}, t_{i} }
    + b (x_{t_{i-1}}) (t_{i}- t_{i-1})
    \Bigr)
    \nonumber \\
     &
    \qquad\qquad
    +
    \sum_{i=1}^N O (|t_i- t_{i-1}|^{3\alpha})
    \nonumber \\
     & \to
    \int_s^t  \sigma (x_u) d{\bf w}_u + \int_s^t  b(x_u) du
  \end{align*}
  as  $|{\mathcal P}|$ tends to zero.
  Note that we have used $3\alpha >1$.
  Thus, we have shown (A).
\end{proof}

Due to this proposition,
Condition (B) can alternatively be adopted as a definition
of RDE \eqref{RDE1}.
In that case, the uniqueness of solution may not be obvious at first sight.
However, if $V_i~(0 \le i \le d)$ is $C^3$, the uniqueness
holds for the following reason.
If both $x$ and $\tilde{x}$  satisfy Condition (B) and
$x_a= \tilde{x}_a$,
then $(x, \sigma (x))$ and $(\tilde{x}, \sigma (\tilde{x}))$
solves the RDE in the sense of controlled path theory.
Using the uniqueness of solution in controlled path theory,
we have $(x, \sigma (x)) =(\tilde{x}, \sigma (\tilde{x}))$ and therefore
$x=\tilde{x}$.

\begin{remark}
  By Proposition \ref{pr.18r^d}, Condition (B)
  can be used as an alternative definition of RDEs.
  This type of formulation
  is a variant of Davie's formulation
  and
  was first given by Bailleul \cite{Bai}.
  This definition naturally carries over to the case
  of  manifold-valued RDEs and has the following
  clear advantages (compared to other definitions of
  manifold-valued RDEs):
  \begin{itemize}
    \item
          A solution is a usual path. It is clear what a manifold-valued path is.
          In contrast, in Lyons' and Gubinelli's definitions,
          a solution has ``higher objects."
          On a manifold, however, these higher objects
          do not look very nice although they can be defined.
    \item
          Clearly, it does not depend on the choice of local coordinate.
          The differential structure of the manifold is
          reflected in the class of the test function $f$.
    \item
          No other structure like a connection on the manifold is needed.
  \end{itemize}
  For other papers along this line of research, see \cite{CW, Dr1, Dr2}. In the next section we will define RDEs on a foliated space
  in a parallel way.
\end{remark}

\begin{definition}\label{def.RDEdom}
  Let $x \in C([0,T], B_1)$,
  where $B_1$ is an open subset of $\R^p$,
  and let $V_i$ be a vector fields on $B_1$ ($0 \le i \le d$).
  Choose $\chi \in C^{\infty} (\R^p, \R)$ with compact support
  such that ${\rm supp}~\chi \subset B_1$ and $\chi \equiv 1$
  on a neighborhood of ${\rm Image}~x :=\{x_t\colon t \in [0,T] \}$.
  Set $V_i^* = \chi V_i$ and regard it as a vector field on
  $\R^p$ ($0 \le i \le d$).
  We say that $x$ is a solution to RDE \eqref{RDE1}
  if it solves
  \[
    dx_{t}=\sum_{i=1}^d V_i^* (x_{t})dw^i_{t} + V_0^* (x_{t}) dt.
  \]
  It is easy to see
  that this definition is independent of the choice of $\chi$.
\end{definition}

Define the time-reversal map ${\mathcal R}_T$ of
an ${\mathbb R}^d$-valued $w$ path with $w_0 =0$
by $({\mathcal R}_T w)(t) = w_{T-t} -w_T$.
Then, ${\mathcal R}_T$ is an involution (i.e., ${\mathcal R}_T^2$
is the identity map)
and
preserves the H\"older norm of $w$.
By slightly abusing the notation,
we also define ${\mathcal R}_T \colon
  G\Omega_{\alpha}([0, T], \R^{d}) \to G\Omega_{\alpha}([0, T], \R^{d})$
by
\[
  ({\mathcal R}_T {\bf w})^{1,i}_{s,t} = -  {\bf w}^{1,i}_{T-t, T-s},
  \quad
({\mathcal R}_T {\bf w})^{2,ij}_{s,t} =  {\bf w}^{2,ji}_{T-t, T-s}.
\]
Then, ${\mathcal R}_T$ is an involution and
preserves the $\alpha$-H\"older rough path norm of ${\bf w}$.
It is easy to see that
${\mathcal R}_T {\bf L} (w)= {\bf L}( {\mathcal R}_T w)$
for a Lipschitz path $w$ with $w_0 =0$.
Hence,
$\{ {\bf L} (w_n)\}_{n=1}^\infty$ converges to ${\bf w}$ in $G\Omega_{\alpha}([0, T], \R^{d})$
if and  only if $\{ {\bf L} ({\mathcal R}_T w_n) \}_{n=1}^\infty$
converges to ${\mathcal R}_T{\bf w}$ in $G\Omega_{\alpha}([0, T], \R^{d})$.

Consider RDE \eqref{RDE1} and assume that $V_i$'s are $C^3_b$.
If ${\bf w} ={\bf L} (w)$
for $w \in C^{1-{\rm Hld}}_0 ([0,T], \R^d)$,
then it is well-known that $\xi \mapsto \Psi (\xi, {\bf w}, {\bf V})_T$
and
$\xi \mapsto \Psi (\xi, {\mathcal R}_T{\bf w}, \check{{\bf V}})_T$
are the inverses of each other.
Here, we set $\check{{\bf V}} =[\tilde{V}_1, \ldots, \tilde{V}_d; -\tilde{V}_0]$
(the sign of the drift was changed). See \cite[Section 8.9]{FH} or \cite[Section 11.2]{FV}.

Hence, $\xi \mapsto \Psi (\xi, {\bf w}, {\bf V})_T$ is a
homeomorphism of $\R^p$.
We can show that  this property also holds for a general
geometric rough path ${\bf w}$
by approximating it by (the natural lift of)  a Lipschitz path.

Now we discuss the differentiability
of the flow $\xi \mapsto x_{t} = \Psi (\xi, {\bf w}, {\bf V})_t$
associated with RDE \eqref{RDE1}.
It is already known that if $V_i$'s are of $C^{k+3}_b$,
then $\xi \mapsto \Psi (\xi, {\bf w}, {\bf V})_t$ is  a
$C^{k}$-diffeomorphism
($k\ge 1$)
and the derivatives satisfy certain linear RDEs.
See \cite[Section 11.2]{FV} for example.
Now we quickly recall this fact.

First, we fix the notation.
The standard gradient on $\R^{p}$ is denoted by $\nabla$.
Hence, if $F\colon \R^{p} \to \R^q$ is sufficiently nice,
then $\nabla^k F \colon \R^{p} \to {\mathcal L}^{k}
  (\R^{p} \times\cdots \times \R^{p},  \R^{q})$,
where
${\mathcal L}^{k}
  (\R^{p} \times\cdots \times \R^{p},  \R^{q})$
stands for the set of $k$-linear maps
from $\R^{p} \times\cdots \times \R^{p}$ to $\R^{q}$.

We set
$J^{(1)}_t (\xi) = \nabla \Psi (\xi, {\bf w}, {\bf V})_t$,
$J^{(-1)}_t (\xi) = [J^{(1)}_t (\xi)]^{-1}$,
where the inverse is taken as a $p\times p$-matrix.
We also set
$J^{(l)}_t (\xi) =\nabla^l \Psi (\xi, {\bf w}, {\bf V})_t$
for $l \ge 2$.
(When necessary, we write $J^{(l)}_t (\xi)
  = J^{(l)} (\xi, {\bf w}, {\bf V})_t$, etc.)

The derivative and its inverse satisfy the following
system of RDEs:
\begin{align}
  \label{RDE_J1}
  dJ^{(1)}_t (\xi)  & =
  \sum_{i=1}^d \nabla \tilde{V}_i (x_{t})  J^{(1)}_t (\xi)  dw^i_{t}
  + \nabla \tilde{V}_0 (x_{t}) J^{(1)}_t (\xi) dt,
  \quad
  J^{(1)}_0={\rm Id},
  \\
  \label{RDE_J2}
  dJ^{(-1)}_t (\xi) & =
  -\sum_{i=1}^d  J^{(-1)}_t (\xi) \nabla \tilde{V}_i (x_{t})  dw^i_{t}
  - J^{(-1)}_t (\xi) \nabla \tilde{V}_0 (x_{t}) dt,
  \quad
  J^{(-1)}_0={\rm Id}.
\end{align}
Here, $\nabla \tilde{V}_i$, $J^{(\pm 1)}_t$ are considerded to be
$p\times p$ matrix-valued.
Given the solution of \eqref{RDE1},
RDEs \eqref{RDE_J1}--\eqref{RDE_J2} are linear RDEs.
So, the system of RDEs  \eqref{RDE1}, \eqref{RDE_J1},
\eqref{RDE_J2} has a unique global solution.

For $l \ge 2$,
$J^{(l)}_t (\xi)$ are known to satisfy
simple (inhomogeneous) linear RDEs
with some kind of triangular structure.
\begin{align}
  \label{RDE_Jk}
  dJ^{(l)}_t (\xi) & =
  \sum_{i=1}^d \nabla \tilde{V}_i (x_{t})  J^{(l)}_t (\xi)  dw^i_{t}
  + \nabla \tilde{V}_0 (x_{t}) J^{(l)}_t (\xi) dt
  \\
                   & \quad +
  \mbox{[terms involving $dw^i_t, dt, x_{t},
          J^{(1)}_t (\xi), \ldots, J^{(l-1)}_t (\xi)$]},
  \quad
  J^{(l)}_0=0.
  \nonumber
\end{align}
The precise form of  \eqref{RDE_Jk} is obtained by
formal differentiation of RDE \eqref{RDE1}.
When $l=2$, for instance,
RDE \eqref{RDE_Jk} reads:
\begin{align}
  \label{RDE_Jk=2}
  dJ^{(2)}_t (\xi) & =
  \sum_{i=1}^d \nabla \tilde{V}_i (x_{t})  J^{(2)}_t (\xi)  dw^i_{t}
  + \nabla \tilde{V}_0 (x_{t}) J^{(2)}_t (\xi) dt
  \\
                   & \quad +
  \sum_{i=1}^d
  \nabla^2 \tilde{V}_i (x_{t})  \langle J^{(1)}_t (\xi),
  J^{(1)}_t (\xi) \rangle dw^i_{t}
  \nonumber                  \\
                   & \quad
  +
  \nabla^2 \tilde{V}_0 (x_{t})  \langle J^{(1)}_t (\xi),
  J^{(1)}_t (\xi) \rangle dt,
  \qquad
  J^{(2)}_0=0.
  \nonumber
\end{align}
Here, $\nabla^2 \tilde{V}_i (x_{t})  \langle J^{(1)}_t (\xi),
  J^{(1)}_t (\xi) \rangle
$
denotes the bilinear map
\[
  \R^{p}\times \R^{p}\ni (y, \tilde{y})\mapsto
  \nabla^2 \tilde{V}_i (x_{t}) \langle J^{(1)}_t (\xi) y, J^{(1)}_t (\xi)\tilde{y}\rangle\in \R^{p}.
\]
Due to the triangular structure, RDE \eqref{RDE_Jk}
has a unique global solution and it admits a
Duhamel-type expression in the sense of rough path integral:
\[
  J^{(l)}_t (\xi)
  =
  J^{(1)}_t (\xi)
  \int_0^t  J^{(-1)}_s (\xi)  \mbox{[terms involving $dw^i_s, ds, x_s,
          J^{(1)}_s (\xi), \ldots, J^{(l-1)}_s (\xi)$]}.
\]
Therefore,
the system of RDEs \eqref{RDE1}, \eqref{RDE_J1},
\eqref{RDE_J2}, \eqref{RDE_Jk} with $2 \le l \le k$
has a unique global solution.
Once we know a unique global solution exists,
we can prove Lyons' continuity theorem for
this system of RDEs by a standard cut-off argument.
This is summarized as follows
(this is essentially shown in \cite[Section 10.7]{FV}):

\begin{proposition}\label{pr.cont_Jk}
  Let the notation be as above and $k \ge 1$.
  Assume
  that the vector fields $V_i~(0 \le i \le d)$ are of $C^{k+3}_b$.
  Then, the unique global solution of
  the system of RDEs \eqref{RDE1}, \eqref{RDE_J1},
  \eqref{RDE_J2}, \eqref{RDE_Jk} with $2 \le l \le k$
  satisfies the following estimate:
  If
  \[
    \max_{i=1,2} \|{\bf w}^i \|_{i\alpha}^{1/i} \le K, \quad
    \max_{i=1,2} \|\hat {\bf w}^i \|_{i\alpha}^{1/i} \le K,\quad
    \| {\bf V}\|_{C^{k+3}_b} \le K, \quad
    \| \hat{\bf V}\|_{C^{k+3}_b} \le K
  \]
  for $K>0$, then there exists a positive constant
  $C^\prime_K$ which depends only on $K$ such that
  \[
    \|   J^{(l)}  (\xi, {\bf w}, {\bf V})_t
    -  J^{(l)} (\hat\xi, \hat{\bf w}, \hat{\bf V})_t
    \|_{\alpha} \le C^\prime_K (|\xi -\hat\xi| +
    {\rm d}_{\alpha}({\bf w},  \hat{\bf w})  +\|{\bf V} -\hat{\bf V}\|_{C^{k+3}_b})
  \]
  for $l=-1, 1,2, \ldots, k$.
  Here, the norm on the left hand side stands for the
  $\alpha$-H\"older seminorm.
\end{proposition}

\begin{remark}
  Let $F\colon \R^p \to \R^p$ be $C^{k}$ and assume that $\nabla F (\xi) = :J_F (\xi)$ is invertible at every $\xi$
  (recall that $J_F (\xi)^{-1} = J_{F^{-1}} (F(\xi))$).
  Let $k \ge 2$.
  Then, by standard argument
  $\nabla^{k} F^{-1}$ can be written in terms of
  $F, F^{-1}$,
  $J_F$ and $\nabla^{l} F$ for $1 \le l \le k$.
  For example, if we denote $G (\xi) = F^{-1} (\xi)$,
  \[
    (\nabla^{2} G) (F(\xi))
    =
    - J_F (\xi)^{-1} \circ (\nabla^2 F )(\xi)
    \langle J_F (\xi)^{-1} \star,  \, J_F (\xi)^{-1} \bullet \rangle.
  \]
  For this reason,
  we need not compute RDEs concretely
  for
  $\nabla^k [ \Psi (\bullet, {\bf w}, {\bf V})^{-1}_t]$ for $k \ge 2$.
\end{remark}

For the rest of this section, we introduce Brownian rough path.
Let $\mu$ be the $d$-dimensional Wiener measure,
that is, the law of the standard $d$-dimensional Brownian motion
starting at $0$.
It sits on $C_0 ([0,T], \R^d)$.
Set
\begin{equation}\label{def.nullset}
  {\mathcal N}
  =\{
  w \in C_0 ([0,T], \R^d) \colon  \mbox{$\{{\bf W}(m) \}_{m=1}^{\infty}$
  is not Cauchy in $G\Omega_{\alpha}([0, T], \R^{d})$}
  \}
\end{equation}
where $w(m)$ is the piecewise linear approximation of $w$
associated with  the dyadic partition $\{ jT/2^m \}_{0 \le j\le 2^m}$
and ${\bf W}(m) = {\bf L} (w(m))$.
It is known that with respect to the Wiener measure $\mu$,
$\mu ({\mathcal N})=0$.
Hence, under $\mu$,
${\bf W} := \lim_{m \to \infty} {\bf W}(m)$ defines a random variable
that takes values in $G\Omega_{\alpha}([0, T], \R^{d})$.
We call it Brownian rough path.

The Stratonovich-type SDE corresponding to RDE \eqref{RDE1} reads
\begin{equation}\label{SDE1}
  dX_{t}=\sum_{i=1}^d V_i (X_{t}) \circ dw^i_{t} + V_0 (x_{t}) dt,
\end{equation}
where $(w_t)$ stands for the canonical realization of
the $d$-dimensional Brownian motion.
When we specify the initial condition we write $X_t =X(t, \xi, w)$
(the dependence on $V_i$'s is suppressed).

If $V_i$'s are $C^3_b$, then
Lyons' continuity theorem (Proposition \ref{pr.cont_LI})
and Wong-Zakai's approximation theorem
imply that
\begin{equation}\label{eq.WZ}
  \Psi (\xi, {\bf W}, {\bf V})=\lim_{m \to \infty}
  \Psi (\xi, {\bf W} (m), {\bf V}) =X(\bullet, \xi, w)
  \qquad
  \mbox{in $C^{\alpha-{\rm Hld}} ([0,T], \R^p)$, $\mu$-a.s.}
\end{equation}
Thus, the solution of SDE is obtained
as a {\it continuous} image of ${\bf W}$.

Next, consider the case that $V_i$'s are only $C^3$.
Under this condition
RDE \eqref{RDE1} may not have
a global solution for some ${\bf w}$.
If we assume that
a unique solution to SDE \eqref{SDE1} does not explode,
then we can prove that
\eqref{eq.WZ} still holds by standard cut-off technique.
(In particular, $\mu (\{w \in {\mathcal N}^c \colon
  \mbox{$\Psi (\xi, {\bf W}, {\bf V})$ does not explode} \}) =1$.)

\section{RDE on foliated space}\label{sec.RDE on foliated spaces}

Now we define RDE on a compact foliated space ${\mathcal M}$.
For  leafwise $C^3$-vector fields $V_i~(0 \le i \le d)$
on ${\mathcal M}$
and
${\bf w}\in G\Omega_{\alpha}([0, T], \R^{d})$, we consider
the following RDE (with drift)
on the time interval $[a,b] \subset [0,T]$:
\begin{equation}\label{RDE2}
  dx_{t}=\sum_{i=1}^d V_i (x_{t})dw^i_{t} + V_0 (x_{t})dt.
\end{equation}

\begin{definition}\label{def.181000}
  A continuous path $x \in C([a,b], {\mathcal M})$ is
  said to be a solution to RDE \eqref{RDE2}
  with initial condition $m \in {\mathcal M}$
  if the following conditions are satisfied:
  \begin{enumerate}
    \item
          $x_{a}=m$.
    \item
          For every leafwise $C^3$-function $f$ on ${\mathcal M}$, it holds that
          \begin{align}
            f(x_{t})-f(x_{s})
             & = \sum_{i=1}^{d}V_{i}f(x_{s}){\bf w}^{1, i}_{s, t}
            +\sum_{j,k=1}^{d}V_{j}V_{k}f(x_{s}){\bf w}^{2, jk}_{s, t}
            \nonumber                                             \\
             & \qquad
            +V_{0}f(x_{s})(t-s)
            +O(|t-s|^{3\alpha})
            \label{eq:190104-1}
          \end{align}
          for $(s, t)  \in \triangle_{[a,b]}$.
          (Here, the estimate of $O(|t-s|^{3\alpha})$ may depend on
          $f$, $[a,b]$, ${\bf w}$, $x$ and $V_i$'s,
          but not on (s,t). See Remark \ref{rem.190920}.)
  \end{enumerate}
  If $x$ solves RDE \eqref{RDE2} for a certain
  initial condition $m$,
  $x$ is simply said to be a solution to RDE \eqref{RDE2}.
\end{definition}

It is obvious that if $x \in C([a,b], {\mathcal M})$ is a solution
to RDE \eqref{RDE2}
as in Definition \ref{def.181000} above
and $[\sigma, \tau] \subset [a,b]$, then its restriction
$x{\restriction_{[\sigma, \tau]}}$ to $[\sigma, \tau]$ is again a solution
to RDE \eqref{RDE2}.

As one can easily expect, a solution to RDE  \eqref{RDE2}
never gets out of the leaf that contains the initial point.
\begin{lemma}\label{lem.121201}
  Let $x \in C([a,b], {\mathcal M})$ be
  a solution to RDE \eqref{RDE2}
  with initial condition $m \in {\mathcal M}$ as in Definition
  \ref{def.181000}.
  Then, $x$ stays in one leaf, that is,
  $x_{s} \in {\mathcal L}_m$ for every $s \in [a,b]$.
  More precisely, if $x \in C([a,b], {\mathcal M})$
  is a solution that stays in a local chart,
  then it stays in one plaque.
\end{lemma}

\begin{proof}
  Let $x \in C([a,b], {\mathcal M})$
  be a solution that stays in a local chart
  $\phi \colon U \simeq B_1 \times B_2$
  and
  write $\phi (x_t) =(y_t, z_t)$.
  Here, $B_1$ and $B_2$ are open sets of
  ${\mathbb R}^p$ and ${\mathcal Z}$, respectively.
  Assume that $(x_t)$ gets out of the plaque indexed by $z_a$
  and set $\sigma = \inf\{ t \in [a,b] \colon z_{t} \neq z_a \}
    \in [a, b)$.
  We will prove the statement by contradiction.
  Without loss of generality, we may and do assume $\sigma=a$.

  Set $M_t = \max_{a \le s \le t} {\rm d}_{{\mathcal Z}} (z_a, z_s)$.
  Then, $M$ is continuous and non-decreasing.
  For sufficiently large $j \in {\mathbb N}$,
  set $a_j = \inf\{ t \in (a,b] \colon  M_t =2^{-j} \}$.
  Then, $\{a_j\}$ is strictly decreasing to $a$
  and $M_{a_j} =2^{-j}= {\rm d}_{{\mathcal Z}} (z_a, z_{a_j})$.
  Define a continuous and strictly increasing function $N$
  by linearly interpolating $(a_{j+1}, 2^{-(j+1)})$ and $(a_j, 2^{-j})$
  and setting $N_a=0$.
  Then, $N^{-1}$ is also continuous and strictly increasing.
  If we set $h(z) = |N^{-1}({\rm d}_{{\mathcal Z}} (z_a, z)) -a|^{\alpha /2}$,
  it defines a bounded continuous function near $z_a$.
  Choose a smooth, compactly supported function $g$ on $B_1$
  such that $g(y) \equiv 1$ near $y_a$.
  Take a leafwise smooth function $f$
  such that $f \circ \phi$ agrees with
  $g(y)h(z)$ near $x_a=\phi^{-1}(y_a, z_a)$.
  Then, for sufficiently large $j$,
  $|f(x_{a_j})-f(x_a)|/(a_j -a)^{\alpha} =(a_j -a)^{-\alpha/2} \nearrow \infty$.
  This implies that $f(x_{\cdot})$ cannot be $\alpha$-H\"older continous,
  which contradicts the assumption that $x$ solves the RDE.
\end{proof}

\begin{lemma}\label{lem.RDEloc}
  Assume that $x \in C([a, b], {\mathcal M})$
  stay in a  coordinate chart $U$ with  chart map
  $\phi\colon U \simeq B_1 \times B_2$,
  where $B_1$ and $B_2$ are open sets of
  ${\mathbb R}^p$ and ${\mathcal Z}$, respectively.
  Write $\phi (x_t) =(y_t, z_t)$.
  Then, $x$ solves RDE \eqref{RDE2} on ${\mathcal M}$ if and only if
  $z_t$ is constant on $[a,b]$ and
  $(y_t)_{t \in [a,b]}$ solves the following RDE on $B_1$
  in the sense of Definition \ref{def.RDEdom}:
  \[
    dy_{t}=\sum_{i=1}^d V_i (y_{t}, z_a)dw^i_{t} + V_0 (y_{t}, z_a) dt.
  \]
  Here, $V_i$ is regarded as a vector filed on $B_1 \subset \R^p$
  in the natural way ($0 \le i \le d$).
\end{lemma}

\begin{proof}
  By Lemma \ref{lem.121201} and
  the way the RDEs are defined, this lemma is almost obvious.
  So, the proof is omitted.
\end{proof}

\begin{notation}\label{def.coordinate}
  For $x \in C([a, b], {\mathcal M})$ and
  a partition $a = t_0 < t_1 < \cdots < t_n =b$  such that
  $(x_t)_{t \in [t_{j-1}, t_j]}$ is contained in a coordinate chart
  $U_j$ for each $j~(1 \le j \le n)$,
  we use the following notation.
  \begin{itemize}
    \item
          The chart map is denoted by
          $\phi_j \colon U_j \simeq B_{j,1} \times B_{j,2}$
          for each $j$.
    \item
          We write $\phi_j (x_t) =(y_t^{[j]}, z_{t}^{[j]})$ for $t \in [t_{j-1}, t_j]$
          for each $j$.
    \item
          $\chi_j \colon {\mathbb R}^p \to [0,\infty)$
          is a smooth, compactly supported function
          such that
          ${\rm supp}~\chi_j \subset B_{j,1}$
          and $\chi_j \equiv 1$ on a certain neighborhood $O_j$ of
          ${\rm Image}~y^{[j]}=\{ y_t^{[j]} \colon t \in [t_{j-1}, t_j] \}$.
    \item
          Set $U_{i}( \,\cdot\, ;z_j) = \chi_j ( \,\cdot\,) V_i ( \,\cdot\,,z_j)$
          for $z_j \in B_{j,2}$ and $0 \le i \le d$.
          If $V_i$ is leafwise $C^k$, then $U_{i}( \,\cdot\, ;z_j)$
          is a $C^k$-vector field on $\R^p$ with compact support
          ($0 \le k \le \infty$).
          (Recall that every continuous function on a compact metric space
          is uniformly continuous.
          Hence, if $V_j$ is leafwise $C^k$ then
          \begin{equation}\label{eq.190123-1}
            \| \tilde{U}_{i}( \,\cdot\, ;z_j^{\prime}) - \tilde{U}_{i}( \,\cdot\, ;z_j) \|_{C^k_b}
            \to 0 \qquad \mbox{as $z_j^{\prime}\to z_j$ in $B_{j,2}$}
          \end{equation}
          for all $0 \le k < \infty$.)
  \end{itemize}
\end{notation}

\begin{lemma}\label{lm.prolong}
  Let $0\le a<c<b \le T$.
  Suppose that $(x_t)_{t \in [a,b]}$ is
  a continuous path in ${\mathcal M}$
  such  that $(x_t)_{t \in [a,c]}$ and $(x_t)_{t \in [c, b]}$
  solve RDE \eqref{RDE2} on $[a,c]$ and $[c,b]$, respectively.
  Then,  $(x_t)_{t \in [a,b]}$ solves RDE \eqref{RDE2} on $[a,b]$.
\end{lemma}

\begin{proof}
  We will show the following three inequalities
  on every subinterval of $[a, b]$:
  Inequality
  \eqref{eq:190104-1} for $f \in C^3_L ({\mathcal M})$ and
  \begin{align}
    f(x_{t})-f(x_{s}) & =\sum_{i=1}^{d}V_{i}f(x_{s}){\bf w}^{1, i}_{s, t}
    +O(|t-s|^{2\alpha})
                      & \mbox{for $f \in C^2_L ({\mathcal M})$,}
    \label{eq:190104-2}
    \\
    f(x_{t})-f(x_{s}) & =O(|t-s|^{\alpha})
                      & \mbox{for $f \in C^1_L ({\mathcal M})$.}
    \label{eq:190104-3}
  \end{align}

  Take a partition $a = t_0 < t_1 < \cdots < t_n =b$
  as in Notation \ref{def.coordinate}.
  We use the notation defined there.
  Moreover, we may assume that
  $t_j =c$ for some $j~(1 \le j \le n-1)$.

  On each $[t_{j-1}, t_j]$,  showing
  \eqref{eq:190104-1}, \eqref{eq:190104-2} and
  \eqref{eq:190104-3}
  reduces to the same problem for the corresponding RDE on $\R^p$
  due to Lemma \ref{lem.RDEloc},
  but it was already
  done in the proof of Proposition \ref{pr.18r^d}.
  (See \eqref{eq.1019-2} and \eqref{eq.1019-3}.)

  Now, we prove by mathematical induction that
  Inequalities \eqref{eq:190104-1}--\eqref{eq:190104-3}
  hold on $[a, t_j]$ for every $1 \le j \le n$.
  The case $j=1$ has just been verified.
  We assume the case $j-1$ and will verify the case $j$
  by using Chen's identity for ${\bf w}$.
  It is obvious that if \eqref{eq:190104-3} holds on
  both $[a, t_{j-1}]$ and $[t_{j-1}, t_j]$, then it holds on $[a, t_{j}]$, too.
  Using  \eqref{eq:190104-3} and
  ${\bf w}^1_{s, t} = {\bf w}^1_{s, u} + {\bf w}^1_{u, t}$
  with $u = t_{j-1}$, we can show
  a similar property for Inequality \eqref{eq:190104-2}.

  Let $a \le s \le t_{j-1} \le t \le t_j$ and set $u:=t_{j-1}$.
  Since  \eqref{eq:190104-1} holds on
  both $[a, t_{j-1}]$ and $[t_{j-1}, t_j]$, we have
  \begin{align}
    f(x_t) - f(x_s)
     & =
    \sum_{i}\{V_{i}f(x_{u}) -  V_{i}f(x_{s})\} {\bf w}^{1, i}_{u, t}
    +
    \sum_{i} V_{i}f(x_{s})  ({\bf w}^{1, i}_{u, t}+ {\bf w}^{1, i}_{s, u})
    \nonumber \\
     & \quad
    +
    \sum_{j,k}\{V_{j}V_{k}f(x_{u}) -  V_{j}V_{k}f(x_{s})\} {\bf w}^{2, jk}_{u, t}
    \nonumber \\
     & \quad
    +
    \sum_{j,k}V_{j}V_{k}f(x_{s})  ({\bf w}^{2, jk}_{u, t}+ {\bf w}^{2, jk}_{s,u})
    \nonumber
    \\
     & \quad
    +
    V_{0}f(x_{s}) (t-s)  + \{V_{0}f(x_{u}) -  V_{0}f(x_{s})\} (t-u)
    + O (|t-s|^{3\alpha}).
    \nonumber
  \end{align}
  The third and the sixth
  terms are $O (|t-s|^{3\alpha})$ due to \eqref{eq:190104-3}.
  By Chen's identity,
  the second term equals $\sum_{i} V_{i}f(x_{s})  {\bf w}^{1, i}_{s, t}$
  and
  the fourth term equals
  $\sum_{j,k}V_{j}V_{k}f(x_{s})
    ({\bf w}^{2, jk}_{s, t}- {\bf w}^{1, j}_{s, u}{\bf w}^{1, k}_{u, t})$.
  Due to \eqref{eq:190104-2}, the first term equals
  $\sum_{j,i}V_{j}V_{i}f(x_{s})
    {\bf w}^{1, j}_{s, u}{\bf w}^{1, i}_{u, t} +O (|t-s|^{3\alpha})$.
  Thus, we have verified \eqref{eq:190104-1} on $[a, t_{j}]$.
  This completes the proof.
\end{proof}

\begin{proposition}\label{pr.gl.sol}
  Let ${\bf w}\in
    G\Omega_{\alpha}([0, T], \R^{d})$ and
  $[a,b]\subset [0,T]$.
  Assume that $V_i~(0 \le i \le d)$ are leafwise $C^3$.
  Then, for every $m \in M$,
  there exists a unique solution $(x_t)_{t \in [a,b]}$ to RDE \eqref{RDE2}
  with $x_a =m$.
  Moreover, if ${\bf w}$ is the natural lift of
  $w \in C^{1-{\rm Hld}}_0 ([0,T], {\mathbb R}^d)$,
  then the unique solution to RDE \eqref{RDE2}
  coincides with that of ODE in the Riemann-Stieltjes sense.
\end{proposition}

\begin{proof}
  For simplicity, we prove the case $[a,b]=[0,T]$ only.
  For every $S \in [0,T)$ and $m \in {\mathcal M}$,
  there exists a unique local solution $(x_t)$ such that
  $x_S =m$.
  We can verify this by
  taking a local chart around $m$ and
  reducing the problem
  to the corresponding one on ${\mathbb R}^p$.
  The uniqueness is easy to see since it is a time-local issue.

  Let us prove the global existence.
  We claim
  that, if $(x_t)$ solves the RDE on $[0, S)$
  (i.e., it solves on $[0, S -\epsilon]$ for every small $\epsilon >0$),
  then $\lim_{t \nearrow S} x_t$ exists
  and $(x_t)_{t \in [0,S]}$ solves the RDE on $[0, S]$.
  (Recall that for an RDE on ${\mathbb R}^p$ with $C^3_b$-coefficients,
  this property is known.)

  Now we give a proof of the above claim.
  Since ${\mathcal M}$ is compact, there exist
  $s_1 < s_2 < \cdots \nearrow S$ such that $\{x_{s_n} \}_{n=1}^\infty$
  converges to some $\hat{m} \in {\mathcal M}$.
  Take a local chart $\phi\colon U \simeq B_1 \times B_2$
  around $\hat{m}$,
  where $B_1$ and $B_2$ are open sets of
  ${\mathbb R}^p$ and ${\mathcal Z}$, respectively.
  We write $\phi (\hat{m})=(\hat{y}, \hat{z})$.
  We also write $\phi (x_t) =(y_t, z_t)$ when $x_t \in U$
  and, in particular, $\phi (x_{s_n}) =(y_{s_n}, z_{s_n})$.

  Let $O$ and $O'$ be bounded
  open subsets of ${\mathbb R}^p$
  such that
  $\hat{y} \in O \subset O'  \subset B_1$.
  Take a smooth, compactly supported function
  $\chi \colon {\mathbb R}^p \to [0,\infty)$ such that
  ${\rm supp}~\chi \subset B_{1}$
  and $\chi \equiv 1$ on $O'$.
  Set $U_{i}^{[n]} ( \,\cdot\,)= \chi ( \,\cdot\,)V_i ( \,\cdot\,, z_{s_n} )$ so that
  $U_{i}^{[n]}$ is a $C^3_b$-vector field on ${\mathbb R}^p$
  whose $C^3_b$-norms are bounded in $n$.
  Consider the RDE on ${\mathbb R}^p$
  associated with $U_{i}^{[n]}$
  ($0 \le i \le d$).
  Then, its solution has a finite $\alpha$-H\"older seminorm
  which is bounded in $n$ by Proposition \ref{pr.cont_LI}.
  Combining this with
  $y_{s_n} \in O$ for sufficiently large $n$,
  we see that
  $(y_t)_{t \in [s_n, s_{n+1}]} \subset O'$ for sufficiently large $n$.
  So,
  $(y_t)_{t \in [s_n, S)} \subset O'$ for some $n$.
  Lemma \ref{lem.121201} then implies
  $z_{s_n} = \hat{z}$ for sufficiently large $n$.
  In other words, $x_t$ stays in $U$ if $t$ is close enough to $S$
  and belongs to one plaque.
  By basic results for RDEs on ${\mathbb R}^p$,
  $\lim_{t \nearrow S} y_t =\hat{y}$ and
  $(y_t)_{t \in [s_n, S]}$ solves the RDE associated with
  $U_{i}^{[n]}$, which is in fact independent of $n$.
  This in turn implies that
  $(x_t)_{t \in [s_n, S]}$ solves the original RDE on ${\mathcal M}$.
  By Lemma \ref{lm.prolong},
  $(x_t)_{t \in [0, S]}$ solves the original RDE, which proves
  the global existence.

  Now we prove the latter half.
  Let ${\bf w}$ be the natural lift of
  $w \in C^{1-{\rm Hld}}_0 ([0,T], {\mathbb R}^d)$.
  For an equation on ${\mathbb R}^p$,
  the solution
  in the rough path sense and the Riemann-Stieltjes sense
  coincide.
  Hence, by localization and Lemma \ref{lm.prolong},
  solutions in these two senses also
  coincide for an RDE on ${\mathcal M}$.
\end{proof}

\begin{definition}
  The unique solution $x$ of RDE \eqref{RDE2}
  in Proposition \ref{pr.gl.sol} above
  will be denoted by
  $x_{t} = \Phi_{[a,b]} (m, {\bf w})_t$.
  When $[a,b]=[0,T]$, it is simply denoted by $\Phi (m, {\bf w})_t$.
  We call $\Phi_{[a,b]}$ and $\Phi$ the Lyons-It\^o map
  associated with RDE \eqref{RDE2}.
  (When the coefficient vector fields in  RDE \eqref{RDE2}
  are replaced by $[V_1, \ldots, V_d; -V_0]$,
  the corresponding
  Lyons-It\^o map is denoted by $\check{\Phi}_{[a,b]}$ and $\check{\Phi}$.)
\end{definition}

\begin{proposition}\label{pr.LI}
  Keep the same assumptions as in Proposition \ref{pr.gl.sol}.
  The Lyons-It\^o map
  \[
    {\mathcal M} \times G\Omega_{\alpha}([0, T], \R^{d})
    \ni (m,  {\bf w}) \mapsto \Phi_{[a,b]} (m, {\bf w})
    \in
    C([a,b], {\mathcal M})
  \]
  is continuous.
\end{proposition}

\begin{proof}
  Take ${\bf w}$ and $m$ arbitrarily and fix them.
  For $x = \Phi (m, {\bf w})$, let $a = t_0 < t_1 < \cdots < t_n =b$
  be a partition as in Notation \ref{def.coordinate}.

  For $z^{[j]} \in B_{j,2}$,
  the Lyons-It\^o map associated with the coefficient
  $U_{i} (\,\cdot\,; z^{[j]} )$ on $[t_{j-1}, t_j]$
  is denoted by
  $\Psi_{[t_{j-1}, t_j]} (y^{[j]}, {\bf w} ; z^{[j]})_t$,
  where
  ${\bf w}$ is the driving rough path and
  $y^{[j]} \in {\mathbb R}^p$ is the starting point at time $t_{j-1}$.
  By \eqref{eq.190123-1} and Lyons' continuity theorem
  (Proposition \ref{pr.cont_LI}),
  $$
    (y^{[j]}, {\bf w} ; z^{[j]})
    \mapsto
    \Psi_{[t_{j-1}, t_j]} ( y^{[j]}, {\bf w} ; z^{[j]})
    \in
    C([t_{j-1}, t_j], {\mathbb R}^p)
  $$
  is continuous.

  If $\tilde{\bf w}$ and $\tilde{m}$ are close enough to
  ${\bf w}$ and $m$, respectively,
  then
  \[
    \sup_{t_0 \le t \le t_1}
    \bigl|
    \Psi_{[t_{0}, t_1]} (y^{[1]}_{t_0}, {\bf w} ; z^{[1]}_{t_0})_t
    -
    \Psi_{[t_{0}, t_1]} (\tilde{y}^{[1]}_{t_0}, \tilde{\bf w} ; \tilde{z}^{[1]}_{t_0})_t
    \bigr|
  \]
  is small enough.
  In particular, $(\tilde{y}_t)_{t \in [t_0, t_1]}$ also stays in $O_1$.
  This implies that
  \[
    \phi_1 (x_t) = \bigl(
    \Psi_{[t_{0}, t_1]} (y^{[1]}_{t_0}, {\bf w} ; z^{[1]}_{t_0})_t,
    z^{[1]}_{t_0}
    \bigr),
    \quad
    \phi_1 (\tilde{x}_t) = \bigl(
    \Psi_{[t_{0}, t_1]} (\tilde{y}^{[1]}_{t_0}, \tilde{\bf w} ; \tilde{z}^{[1]}_{t_0})_t,
    \tilde{z}^{[1]}_{t_0}
    \bigr),
    \quad t \in [t_0,t_1]
  \]
  and
  \[
    \sup_{t_0 \le t \le t_1} {\rm d}_{{\mathcal M}}
    (
    x_t, \tilde{x}_t
    )
  \]
  is small enough, too.
  In particular,
  $y^{[2]}_{t_1}$ and $\tilde{y}^{[2]}_{t_1}$
  are close enough and so are
  $z^{[2]}_{t_1}$ and $\tilde{z}^{[2]}_{t_1}$.

  Next, we work on  the second interval $[t_1, t_2]$.
  Since $x_{t_1}$ and $\tilde{x}_{t_1}$ are already given,
  we consider
  $\Psi_{[t_{1}, t_2]} (y^{[2]}_{t_1},{\bf w} ; z^{[2]}_{t_1})_t$,
  where $\phi_2(x_{t_1} )=(y^{[2]}_{t_1}, z^{[2]}_{t_1} )$,
  and its corresponding object with ``tilde."
  By the same argument as above,
  if $\tilde{\bf w}$ and $\tilde{m}$ are close enough to
  ${\bf w}$ and $m$, respectively,
  then all arguments of $\Psi_{[t_{1}, t_2]}$ are close enough,
  which in turn implies that
  \[
    \sup_{t_1 \le t \le t_2}
    \bigl|
    \Psi_{[t_{1}, t_2]} (y^{[2]}_{t_1}, {\bf w} ; z^{[2]}_{t_1})_t
    -
    \Psi_{[t_{1}, t_2]} (\tilde{y}^{[2]}_{t_1}, \tilde{\bf w} ; \tilde{z}^{[2]}_{t_1})_t
    \bigr|
  \]
  is small enough.
  Hence, by the same argument we see that
  $(\tilde{y}_t)_{t \in [t_1, t_2]}$ stays in $O_2$ and
  that
  $\sup_{t_0 \le t \le t_2} {\rm d}_{{\mathcal M}}
    \bigl(
    x_t, \tilde{x}_t
    \bigr)
  $
  is small enough.

  Repeating this argument $n$-times, we complete the proof.
\end{proof}

Now we prove that the solution of RDE defines a flow
of
leaf preserving
homeomorphisms on ${\mathcal M}$.

\begin{proposition}\label{pr.homeo}
  Assume that $V_i~(0 \le i \le d)$ are leafwise $C^3$.
  For every $S \in [0,T]$ and ${\bf w}\in
    G\Omega_{\alpha}([0, T], \R^{d})$, the map
  $m \mapsto \Phi (m, {\bf w})_S$ is
  a leaf preserving
  homeomorphism on ${\mathcal M}$
  and
  its inverse is $m \mapsto \check{\Phi} (m, {\mathcal R}_S{\bf w})_S$.
  Moreover, the following map is continuous:
  \[
    G\Omega_{\alpha}([0, T], \R^{d})  \ni
    {\bf w} \mapsto
    \bigl[ t \in [0,T] \mapsto
      \Phi (\bullet, {\bf w})_t \bigr]
    \in
    C([0,T], {\rm Homeo}_L ({\mathcal M}) ).
  \]
\end{proposition}

\begin{proof}
  First, we show that
  $\Phi(\bullet, {\bf w})_{S}\in {\rm Homeo}_{L}(\mathcal{M})$ and $\Phi(\bullet, {\bf w})_{S}^{-1}=\check{\Phi}(\bullet, {\mathcal R}_S {\bf w})_{S}$  for any ${\bf w}\in G\Omega_{\alpha}([0, T], \R^{d})$ and $S\in(0, T]$.
  Take ${\bf w}\in G\Omega_{\alpha}([0, T], \R^{d})$, $S\in(0, T]$. By virtue of Lemma \ref{lem.121201} and Proposition \ref{pr.LI}, we have $\Phi(\bullet, {\bf w})_{S}\in C_{L}(\mathcal{M}, \mathcal{M})$ and
  $\check{\Phi}(\bullet, {\mathcal R}_S {\bf w})_{S}\in C_{L}(\mathcal{M}, \mathcal{M})$. Hence it remains to show that $\Phi(\bullet, {\bf w})_{S}^{-1}=\check{\Phi}(\bullet, {\mathcal R}_S {\bf w})_{S}$. Let $m\in \mathcal{M}$ and $0=t_{0}<t_{1}<\cdots<t_{n}=T$ a partition as in
  Notation \ref{def.coordinate} for $x=\Phi(m, {\bf w})$. Moreover we may assume that
  $t_{\ell}=S$ for some $1\le \ell\le n$.
  From Proposition \ref{pr.LI}, we see that there exists a neighborhood
  $\mathcal{O}_{m}$ of $m$ such that for any $\tilde{m}\in \mathcal{O}_{m}$ and $1\le j\le \ell$, $(\tilde{x}_{t}=\Phi(\tilde{m}, {\bf w})_{t})_{t\in[t_{j-1}, t_{j}]} \subset U_{j}$
  and $(\tilde{y}_{t}^{[j]})_{t\in [t_{j-1}, t_{j}]}\subset O_{j}$, where
  \[
    \phi_{j}(\tilde{x}_{t})=(\Psi_{[t_{j-1}, t_{j}]}(\tilde{y}_{t_{j-1}}^{[j]}, {\bf w}; \tilde{z}_{t_{j-1}}^{[j]})_{t}, \tilde{z}_{t_{j-1}}^{[j]}), \quad t\in [t_{j-1}, t_{j}].
  \]
  Then, the map $\mathcal{O}_{m}\ni \tilde{m}\mapsto \Phi(\tilde{m}, {\bf w})_{S}$ can be represented as
  \[
    \Phi(\bullet, {\bf w})_{S}
    =	\Phi_{[t_{\ell-1}, t_{\ell}]}(\bullet, {\bf w})_{t_{\ell}}\circ \Phi_{[t_{\ell-2}, t_{\ell-1}]}(\bullet, {\bf w})_{t_{k-1}}\circ \cdots \circ\Phi_{[t_{0}, t_{1}]}(\bullet, {\bf w})_{t_{1}}
  \]
  and each $\Phi_{[t_{j-1}, t_{j}]}(\bullet, {\bf w})_{t_{j}}$ has a local representation of the form
  \[
    (y^{[j]}, z^{[j]})\mapsto (\Psi_{[t_{j-1}, t_{j}]}(y^{[j]}, {\bf w}; z^{[j]})_{t_{j}}, z^{[j]})
  \]
  at $\Phi(m, {\bf w})_{t_{j-1}}\in U_{j}$. The local representation above has the inverse
  \begin{equation}\label{eq.loc inverse}
    (y^{[j]}, z^{[j]})\mapsto (\check{\Psi}_{[S-t_{j}, S-t_{j-1}]}(y^{[j]}, \mathcal{R}_{S}{\bf w}; z^{[j]})_{S-t_{j-1}}, z^{[j]}),
  \end{equation}
  where $\check{\Psi}_{[a, b]}$ is the Lyons-It\^o map on $[a, b]$ associated with the coefficients
  \[
    [U_1(\,\cdot\,; z^{[j]}), \ldots, U_d(\,\cdot\,; z^{[j]}); {-U_0(\,\cdot\,;z^{[j]})}].
  \]
  On the other hand, from Lemma \ref{lem.RDEloc},
  the map \eqref{eq.loc inverse} coincides with a
  local representation of $\check{\Phi}_{[S-t_{j}, S-t_{j-1}]}(\bullet, \mathcal{R}_{S}{\bf w})_{S-t_{j-1}}$ at
  $\Phi(m, {\bf w})_{t_{j}}\in U_{j}$. Combining this and the representation
  \[
    \begin{aligned}
      \check{\Phi}(\bullet, \mathcal{R}_{S}{\bf w})_{S}
       & =	\check{\Phi}_{[S-t_{1}, S-t_{0}]}(\bullet, \mathcal{R}_{S}{\bf w})_{S}\circ \check{\Phi}_{[S-t_{2}, S-t_{1}]}(\bullet, \mathcal{R}_{S}{\bf w})_{S-t_{1}}\circ \cdots \\
       & \hspace{1em}\cdots \circ\check{\Phi}_{[S-t_{\ell}, S-t_{\ell-1}]}(\bullet, \mathcal{R}_{S}{\bf w})_{S-t_{\ell-1}}
    \end{aligned}
  \]
  on a neighborhood of $\Phi(m, {\bf w})_{S}$, we conclude that $\check{\Phi}(\Phi(m, {\bf w})_{S}, \mathcal{R}_{S}{\bf w})_{S}=m$.
  In a similar way, we can show that $\Phi(\check{\Phi}(m, \mathcal{R}_{S}{\bf w})_{S}, {\bf w})_{S}=m$ holds and consequently we have $\Phi(\bullet, {\bf w})_{S}^{-1}=\check{\Phi}(\bullet, {\mathcal R}_S {\bf w})_{S}$.

  Now, we show that for each  ${\bf w}$,
  $\bigl[ t  \mapsto
      \Phi ( \bullet, {\bf w})_t \bigr]
    \in
    C([0,T], {\rm Homeo}_L ({\mathcal M}) ).$
  We use the same notation above.
  By Lyons' continuity theorem for an RDE
  with $C^3_b$ coefficients (Proposition \ref{pr.cont_LI}),
  we actually have
  \[
    \sup_{t_{j-1} \le s < t \le t_j}
    |t-s |^{-\alpha}
    \bigl|
    \Psi_{[t_{j-1}, t_j]} ( \tilde{y}^{[j]}_{t_{j-1}}, \tilde{\bf w}
    ; \tilde{z}^{[j]}_{t_{j-1}})_t
    -
    \Psi_{[t_{j-1}, t_j]} (\tilde{y}^{[j]}_{t_{j-1}}, \tilde{\bf w},
    ; \tilde{z}^{[j]}_{t_{j-1}})_s
    \bigr|
    \le
    C
  \]
  for some positive constant $C$ independent of
  $\tilde{\bf w}, \tilde{y}^{[j]}_{t_{j-1}},\tilde{z}^{[j]}_{t_{j-1}}$
  if they
  are sufficiently close to
  ${\bf w}, y^{[j]}_{t_{j-1}}, z^{[j]}_{t_{j-1}}$, respectively.
  (Here, ${\bf w}, y^{[j]}_{t_{j-1}}, z^{[j]}_{t_{j-1}}$ are  fixed.)
  For every $s$, we will let $t \to s$.
  In a similar way as in the proof of Proposition \ref{pr.LI},
  we can see from this the following:
  Around every $m \in {\mathcal M}$,
  we can find a neighborhood ${\mathcal O}^{\prime}_m$
  (which may depend only on $m$ and ${\bf w}$)
  such that
  \[
    \sup_{\tilde{m} \in {\mathcal O}_m^{\prime}}
    {\rm d}_{{\mathcal M}}  ( \Phi ( \tilde{m}, {\bf w})_t,
    \Phi (\tilde{m}, {\bf w}))_s
    \to 0
    \quad
    \mbox{as $t \to s$.}
  \]
  Since ${\mathcal M}$ is compact, we see
  the above convergence is actually uniform over $\tilde{m} \in {\mathcal M}$.
  Hence,
  $t \mapsto \Phi (\bullet, {\bf w})_t \in C ({\mathcal M}, {\mathcal M})$
  is continuous.

  We calculate the inverse flow. The inverse of
  $m \mapsto \Phi (m, {\bf w})_t$ is denoted by $\Phi (\bullet, {\bf w})^{-1}_t$.
  Fix any $s$.
  It is easy to see that
  \begin{equation}\label{eq.180110-3}
    \Phi (\bullet, {\bf w})_t^{-1}
    =
    \begin{cases}
      \check{\Phi}(\bullet,  {\mathcal R}_s {\bf w})_s
      \circ  \check{\Phi} (\bullet,  {\mathcal R}_t {\bf w})_{t-s}
       & \mbox{if $s < t$,} \\
      \check{\Phi}( \bullet,  {\mathcal R}_s {\bf w})_s
      \circ  \Phi (\bullet,   \Sigma_t {\bf w})_{s-t}
       & \mbox{if $s > t$.}
    \end{cases}
  \end{equation}
  Here, we set $(\Sigma_t {\bf w})_{u,v} = {\bf w}_{u-t,v-t}$
  for $t \le u \le v$.
  (A similar property to \eqref{eq.180110-3} holds for $\Psi_{[t_{j-1}, t_j]}$, too.)
  Therefore,
  \[
    {\rm d}_{{\mathcal M}} (\Phi (\mu, {\bf w})^{-1}_t,
    \Phi ( \mu, {\bf w})^{-1}_s)
    =
    {\rm d}_{{\mathcal M}}
    (\check{\Phi}( \mu,  {\mathcal R}_s {\bf w})_s,
    \check{\Phi}(\mu',  {\mathcal R}_s {\bf w})_s),
  \]
  where $\mu'=\check{\Phi}(\mu,  {\mathcal R}_t {\bf w})_{t-s}$
  or $\Phi (\mu, \Sigma_t {\bf w})_{s-t}$.
  ($\mu$ is an arbitrary element of ${\mathcal M}$,
  but in spirit we think $\mu = \Phi (m,  {\bf w})_s$.)
  By Proposition \ref{pr.cont_LI} again,
  for every $\mu$ and $s$,
  there exists a neighborhood ${\mathcal O}_\mu$
  such that
  $\sup_{\tilde{\mu}\in \mathcal{O}_{\mu}}{\rm d}_{{\mathcal M}} (\tilde{\mu}, \tilde{\mu}')
    \to 0$ as $t\to s$.
  Combining this with the following inequality
  \[
    \sup_{t_{j-1} \le s \le t_j}
    \bigl|
    \check{\Psi}_{[t_{j-1}, t_j]} (y^{[j]}_{t_{j-1}}, {\bf w} ; z^{[j]}_{t_{j-1}})_s
    -
    \check{\Psi}_{[t_{j-1}, t_j]} ( \tilde{y}^{[j]}_{t_{j-1}}, {\bf w} ; z^{[j]}_{t_{j-1}})_s
    \bigr|
    \le C | y^{[j]}_{t_{j-1}} - \tilde{y}^{[j]}_{t_{j-1}} |
  \]
  for some constant $C>0$
  independent of $z^{[j]}_{t_{j-1}}$
  ($C$ depends on
  ${\bf w}$ only via $\max_{i=1,2} \|{\bf w}^i \|_{i\alpha}^{1/i}$),
  we can show that
  \[
    \sup_{\tilde{\mu} \in {\mathcal O}_\mu}
    {\rm d}_{{\mathcal M}}  ( \Phi (\tilde{\mu}, {\bf w})^{-1}_t,
    \Phi ( \tilde{\mu}, {\bf w})^{-1}_s)
    \to 0
    \quad
    \mbox{as $t \to s$.}
  \]
  This implies
  the continuity of $t \mapsto \Phi (\bullet, {\bf w})^{-1}_t \in
    C({\mathcal M}, {\mathcal M})$.
  Hence,
  $[t \mapsto \Phi (\bullet, {\bf w})_t ] \in
    C([0,T], {\rm Homeo}_L ({\mathcal M}))$.

  The proof of the continuity of
  ${\bf w} \mapsto
    \bigl[ t  \mapsto
      \Phi ( \bullet, {\bf w})_t \bigr]$ is quite similar.
  So, we only give a sketch of proof.
  Let ${\bf w}(l) \to {\bf w}$ as $l \to \infty$.
  We need to show that
  \begin{eqnarray}
    \lefteqn{
    \sup_{0 \le s \le T} \sup_{m \in {\mathcal M}}
    {\rm d}_{{\mathcal M}}  ( \Phi (m, {\bf w}(l))_s,
    \Phi (m, {\bf w})_s)
    }
    \label{eq.190110-2}
    \\
    &\quad +\sup_{0 \le s \le T} \sup_{\mu \in {\mathcal M}}
    {\rm d}_{{\mathcal M}}  ( \Phi (\mu, {\bf w}(l))^{-1}_s,
    \Phi (\mu, {\bf w})^{-1}_s)
    \to 0
    \quad
    \mbox{as $l \to \infty$.}
    \nonumber
  \end{eqnarray}
  To prove the convergence of the first term in
  \eqref{eq.190110-2},
  it is enough to find, for each $m$ and $s$,
  there exist a neighborhood ${\mathcal U}_{s,m} \subset
    [0,T]\times {\mathcal M}$ such that
  \[
    \sup_{(\tilde{s}, \tilde{m}) \in {\mathcal U}_{s,m} }
    {\rm d}_{{\mathcal M}}  ( \Phi (\tilde{m}, {\bf w}(l))_{\tilde{s}},
    \Phi (\tilde{m}, {\bf w})_{\tilde{s}})
    \to 0
    \quad
    \mbox{as $l \to \infty$.}
  \]
  In the same way as before,
  we can prove this by using Proposition \ref{pr.cont_LI}.
  In a similar way, using \eqref{eq.180110-3},
  we can prove the convergence of the second term in
  \eqref{eq.190110-2}.
\end{proof}

Next we prove that the solution of RDE
\eqref{RDE2} defines a flow of leaf preserving diffeomorphisms on
${\mathcal M}$ if the vector fields $V_{i}$\rq{}s have higher regularity.

\begin{proposition}\label{pr.diffeo}
  Assume that $V_{i}\,(0\le i\le d)$ are
  leafwise $C^{k+3}$-vector field $(k\ge 1)$.
  Then,
  for every $S \in [0,T]$ and ${\bf w}\in
    G\Omega_{\alpha}([0, T], \R^{d})$, the map
  $m \mapsto \Phi (m, {\bf w})_S$ is
  a leaf preserving $C_{L}^{k}$-diffeomorphism on $\mathcal{M}$.
  Moreover, the following map is continuous.
  \[
    G\Omega_{\alpha}([0, T], \R^{d})  \ni
    {\bf w} \mapsto
    \bigl[ t \in [0,T] \mapsto
      \Phi (\bullet, {\bf w})_t \bigr]
    \in
    C([0,T], {\rm Diff}_{L}^{k} ({\mathcal M}) ).
  \]
\end{proposition}

\begin{proof}
  First, we show that
  $\Phi(\bullet, {\bf w})_{S}\in C_{L}^{k}(\mathcal{M}, \mathcal{M})$ for any ${\bf w}\in G\Omega_{\alpha}([0, T], \R^{d})$ and $S\in(0, T]$. Take ${\bf w}\in G\Omega_{\alpha}([0, T], \R^{d})$, $S\in(0, T]$, and $m\in \mathcal{M}$.
  Let $0=t_{0}<t_{1}<\cdots<t_{n}=T$ and $\mathcal{O}_{m}$ be a partition and a neighborhood of $m$ as in the proof of
  Proposition \ref{pr.homeo}.
  The Lyons' continuity theorem for an RDE
  with $C^{k+3}_b$ coefficients (Proposition \ref{pr.cont_Jk}) yields that
  \[
    (y^{[j]}, z^{[j]})\mapsto \Psi_{[t_{j-1}, t_{j}]}(y^{[j]}, {\bf w}; z^{[j]})_{t_{j}}
  \]
  is of class $C_{L}^{k}$.
  Since the map $\mathcal{O}_{m}\ni m\mapsto \Phi(m, {\bf w})_{S}$ can be represented as
  \begin{equation}\label{composition}
    \Phi(\bullet, {\bf w})_{S}
    =	\Phi_{[t_{\ell-1}, t_{\ell}]}(\bullet, {\bf w})_{t_{\ell}}\circ \Phi_{[t_{\ell-2}, t_{\ell-1}]}(\bullet, {\bf w})_{t_{\ell-1}}\circ \cdots \circ\Phi_{[t_{0}, t_{1}]}(\bullet, {\bf w})_{t_{1}},
  \end{equation}
  and each $\Phi_{[t_{j-1}, t_{j}]}(\bullet, {\bf w})_{t_{j}}$ has a local representation of the form
  \[
    (y^{[j]}, z^{[j]})\mapsto (\Psi_{[t_{j-1}, t_{j}]}(y^{[j]}, {\bf w}; z^{[j]})_{t_{j}}, z^{[j]})
  \]
  at $\Phi(m, {\bf w})_{t_{j-1}}\in U_{j}$, we see that $\Phi(\bullet, {\bf w})\in C_{L}^{k}(\mathcal{M}, \mathcal{M})$. It is similarly verified that
  $\Phi(\bullet, {\bf w})_{S}^{-1}=\check{\Phi}(\bullet, {\mathcal R}_S {\bf w})_{S}\in C_{L}^{k}(\mathcal{M}, \mathcal{M})$.
  Hence $\Phi(\bullet, {\bf w})_{S}\in {\rm Diff}_{L}^{k}(\mathcal{M})$.

  Next we show that for each ${\bf w}$, $[t\mapsto \Phi(\bullet, {\bf w})_{t}]\in C([0, T], {\rm Diff}_{L}^{k}(\mathcal{M}))$. We use the same notation as above.
  By Proposition \ref{pr.cont_Jk}, if
  $\tilde{{\bf w}}$, $\tilde{y}_{t_{j-1}}^{[j]}$, $\tilde{z}_{t_{j-1}}^{[j]}$ are sufficiently close to
  ${\bf w}$, $y_{t_{j-1}}^{[j]}$, $z_{t_{j-1}}^{[j]}$,
  we can find a positive constant $C$ independent of $\tilde{{\bf w}}$, $\tilde{y}_{t_{j-1}}^{[j]}$, $\tilde{z}_{t_{j-1}}^{[j]}$ such that
  \[
    \sup_{|\beta|\le k, t_{j-1}\le s<t\le t_{j}}|t-s|^{-\alpha}
    \bigl|
    \partial_{y}^{\beta}\Psi_{[t_{j-1}, t_j]} ( \tilde{y}^{[j]}_{t_{j-1}}, \tilde{\bf w}
    ; \tilde{z}^{[j]}_{t_{j-1}})_t
    -
    \partial_{y}^{\beta}\Psi_{[t_{j-1}, t_j]} (\tilde{y}^{[j]}_{t_{j-1}}, \tilde{\bf w},
    ; \tilde{z}^{[j]}_{t_{j-1}})_s
    \bigr|
    \le
    C.
  \]
  Combining this with the representation (\ref{composition}), we see that there exists a neighborhood $\mathcal{O}^{\prime}_{m}\subset \mathcal{O}_{m}$ of $m$ such that the local representation of
  $\Phi(\bullet, {\bf w})_{t}$ from $\mathcal{O}_{m}^{\prime}$ into  $U_{\ell}$ converges to that of $\Phi(\bullet, {\bf w})_{S}$
  as $t\to S$ in the sense of (\ref{eq.loc convergence})  in Proposition \ref{pr.loc convergence}.
  Hence $\Phi(\bullet, {\bf w})_{t}\to \Phi(\bullet, {\bf w})_{S}$ in $C_{L}^{k}(\mathcal{M}, \mathcal{M})$ as $t\to S$.
  Using the same argument for $\check{\Phi}$ and the equation (\ref{eq.180110-3}), we can also show that
  $\Phi(\bullet, {\bf w})_{t}^{-1}\to \Phi(\bullet, {\bf w})_{S}^{-1}$ in $C_{L}^{k}(\mathcal{M}, \mathcal{M})$ as $t\to S$. The proof of the continuity of $t\in [0, T]\mapsto \Phi(\bullet, {\bf w})_{t}\in {\rm Diff}_{L}^k (\mathcal{M})$ is now complete.

  It remains to prove the continuity of
  ${\bf w}\mapsto [t\in [0, T]\mapsto \Phi(\bullet, {\bf w})_{t}]\in C([0, T], {\rm Diff}_{L}^k(\mathcal{M}))$. Let ${\bf w}\in G\Omega_{\alpha}([0, T], \R^{d})$, $m\in \mathcal{M}$, and ${\bf w}(l)\to {\bf w}$ as $l\to \infty$.
  By virtue of Proposition \ref{pr.LI}, we can take a positive integer $L$, a neighborhood $\mathcal{O}_{m}$ of $m$, and
  a partition $0=t_{0}<t_{1}<\cdots<t_{n}=T$ as in Notation \ref{def.coordinate} such that for any $l\ge L$ and $\tilde{m}\in \mathcal{O}_{m}$,
  $(\tilde{x}_{t}^{l}=\Phi(\tilde{m}, {\bf w}(l))_{t})_{t\in[t_{j-1}, t_{j}]} \subset U_{j}$
  and $(\tilde{y}_{t}^{[j], l})_{t\in [t_{j-1}, t_{j}]}\subset O_{j}$, where
  \[
    \phi(\tilde{x}_{t}^{l})=(\Psi_{[t_{j-1}, t_{j}]}(\tilde{y}_{t_{j-1}}^{[j],l}, {\bf w}(l); \tilde{z}_{t_{j-1}}^{[j], l}), \tilde{z}_{t_{j-1}}^{[j], l}), \quad t\in [t_{j-1}, t_{j}].
  \]
  From Proposition \ref{pr.cont_Jk} again, we have
  \begin{eqnarray*}
    \sup_{|\beta|\le k, t_{j-1}\le t\le t_{j}}
    \bigl|
    \partial_{y}^{\beta}\Psi_{[t_{j-1}, t_{j}]}(\tilde{y}^{[j]}, {\bf w}(l); z^{[j]})_{t}
    -\partial_{y}^{\beta}\Psi_{[t_{j-1}, t_{j}]}(y^{[j]}, {\bf w}; z^{[j]})_{t}
    \bigl|
    \\
    \le
    C(|\tilde{y}^{[j]}-y^{[j]}|+{\rm d}_{\alpha}({\bf w}(l), {\bf w}))
  \end{eqnarray*}
  for some positive constant $C$ independent of $l$, $j$, and $z^{[j]}$. Combining this with the equation (\ref{composition}), we can show that there exists
  a neighborhood $\mathcal{O}_{m}^{\prime}\subset \mathcal{O}_{m}$ of $m$ such that for any $1\le j\le n$, the local representation of $\Phi(\bullet, {\bf w}(l))_{t}$ from $O_{m}^{\prime}$ into $U_{j}$ converges to
  that of $\Phi(\bullet, {\bf w})_{t}$ uniformly in $t\in [t_{j-1}, t_{j}]$ as $l\to \infty$ in the sense of (\ref{eq.loc convergence}) in Proposition \ref{pr.loc convergence}. This implies that ${\bf w}\mapsto [t\in [0, T]\mapsto \Phi(\bullet, {\bf w})_{t}]\in C([0, T], C_{L}^{k}(\mathcal{M}, \mathcal{M}))$ is continuous.
  In a similar way, we can also prove the continuity
  of ${\bf w}\mapsto [t\in [0, T]\mapsto \Phi(\bullet, {\bf w})_{t}^{-1}]\in C([0, T], C_{L}^{k}(\mathcal{M}, \mathcal{M}))$. This completes the proof of the proposition.
\end{proof}

\section{Applications of rough path approach}\label{sec.applications}

We finish with four applications of our rough path approach.
As the first one,
we prove the existence of stochastic flow
associated with the following Stratonovich
SDE on ${\mathcal M}$:
\begin{equation}\label{SDE_M}
  dx_{t}=\sum_{i=1}^d V_i (x_{t}) \circ dw^i_{t} + V_0 (x_{t})dt.
\end{equation}

\begin{proposition}\label{pr.version of SDE}
  Assume that $V_i~(0 \le i \le d)$  are leafwise $C^3$-vector fields
  on ${\mathcal M}$.
  Then, $\Phi (m, {\bf W})$ is the
  stochastic flow of leafwise homeomorphisms
  on ${\mathcal M}$ associated with SDE \eqref{SDE_M}.
\end{proposition}

\begin{proof}
  What remains to prove is that, for each fixed $m \in {\mathcal M}$,
  $t \mapsto \Phi (m, {\bf W})_t$ almost surely coincides with
  the unique solution of  SDE \eqref{SDE_M} with $x_0 =m$.
  In Suzaki  \cite{Suz},  it is shown that this SDE has
  a non-exploding solution.
  (Precisely, the vector fields are assumed to be of $C_L^\infty$ in  \cite{Suz}.
  But, a careful reading of his proof reveals that
  assuming $C_L^3$ is enough for this part.)
  Since $m$ is fixed, \eqref{SDE_M} is now an SDE on
  the leaf ${\mathcal L}_m$, which is a smooth manifold
  and hence admits an embedding into a Euclidean space by
  Whitney's theorem.
  By extending $V_i$'s to vector fields on the Euclidean space,
  $(x_t)_{0 \le t \le T}$  can be viewed as
  a non-exploding solution of an SDE on  the Euclidean space
  with $C^3$ coefficients.
  Therefore, $x = \Phi (m, {\bf W})$ holds for almost all $w$ from $\eqref{eq.WZ}$.
\end{proof}

The second application is on the measurability of
the strong solution of SDE \eqref{SDE_M}.
Thanks to rough path theory
we can improve Suzaki \cite[Theorem 2.1]{Suz}
as follows.
Let $C_{L}([0, T], \mathcal{M})$ be the totality of continuous maps $x\colon[0, T]\to \mathcal{M}$
such that 	the image is contained in a single leaf.
Endowed with the compact-open topology
(the topology of uniform convergence),
$C_{L}([0, T], \mathcal{M})$
is a complete, separable, metrizable space. For a topological space $S$, we denote by $\mathcal{B}(S)$ the Borel
$\sigma$-field of $S$.

\begin{proposition}\label{pr.190815}
  Assume that $V_i~(0 \le i \le d)$  are leafwise $C^3$-vector fields
  on ${\mathcal M}$. Then the map
  \[
    \mathcal{M}\times C_{0}([0, T], \R^{d})\ni (m, w)
    \mapsto \Phi(m, {\bf W})\in C_{L}([0, T], \mathcal{M})
  \]
  is $\mathcal{B}(\mathcal{M})\otimes\overline{\mathcal{B}(C_{0}([0, T], \R^{d}))}^{\mu}/\mathcal{B}(C_{L}([0, T], \mathcal{M}))$-measurable, where
  $\overline{\mathcal{B}(C_{0}([0, T], \R^{d}))}^{\mu}$ is the completion of
  $\mathcal{B}(C_{0}([0, T], \R^{d}))$ by the $d$-dimensional Wiener measure $\mu$.
\end{proposition}

\begin{proof}
  Since Brownian rough path
  \[
    C_{0}([0, T], \R^{d})\ni w\mapsto {\bf W}={\bf W}(w)\in G\Omega_{\alpha}([0, T], \R^{d})
  \]
  is $\overline{\mathcal{B}(C_{0}([0, T], \R^{d}))}^{\mu}/\mathcal{B}(G\Omega_{\alpha}([0, T], \R^{d})$-measurable, the map
  \[
    \mathcal{M}\times C_{0}([0, T], \R^{d})\ni (m, w)
    \mapsto (m, {\bf W})\in \mathcal{M}\times G\Omega_{\alpha}([0, T], \R^{d})
  \]
  is $\mathcal{B}(\mathcal{M})\otimes\overline{\mathcal{B}(C_{0}([0, T], \R^{d}))}^{\mu}/\mathcal{B}(\mathcal{M})\otimes \mathcal{B}(G\Omega_{\alpha}([0, T], \R^{d}))$-measurable.
  By virtue of Proposition \ref{pr.LI}, the Lyons-
  It\^o map
  \[
    (m, {\bf w})\ni \mathcal{M}\times G\Omega_{\alpha}([0, T], \R^{d})\mapsto \Phi(m, {\bf w})\in C_{L}([0, T], \mathcal{M})
  \]
  is $\mathcal{B}(\mathcal{M})\otimes \mathcal{B}(G\Omega_{\alpha}([0, T], \R^{d}))/\mathcal{B}(C_{L}([0, T], \mathcal{M}))$-measurable. Therefore the composition of them is $\mathcal{B}(\mathcal{M})\otimes\overline{\mathcal{B}(C_{0}([0, T], \R^{d}))}^{\mu}/\mathcal{B}(C_{L}([0, T], \mathcal{M}))$-measureble.
  Now the proof is complete.
\end{proof}

\begin{remark}
  In Suzaki \cite[Theorem 2.1] {Suz},
  the map in Proposition \ref{pr.190815} was constructed by the Yamada-Watanabe theorem and shown to be merely
  \[
    \bigcap_{\nu}
    \overline{\mathcal{B}(\mathcal{M})\otimes
    \mathcal{B}(C_{0}([0, T], \R^{d}))}^{\nu \otimes \mu}
    /\mathcal{B}(C_{L}([0, T], \mathcal{M}))
  \]
  -measurable, where $\nu$ runs over
  all  Borel probability measure on $\mathcal{M}$.
\end{remark}

The third and fourth application are
a support theorem
and a large deviation principle.
The proofs use the continuity of Lyons-It\^o map
and are standard.

As usual, let
\[
  {\mathcal H} = \{ h = \int_0^{\cdot} h^{\prime}_s ds
  \in C_0 ([0,T], {\mathbb R}^d)
  \colon
  \| h \|_{{\mathcal H}}^2 := \int_0^T |h^{\prime}_s|^2 ds <\infty
  \}
\]
be Cameron-Martin Hilbert space.
For $l \ge 1$ and $h \in {\mathcal H}$, $h (l)  \in {\mathcal H}$
stands for the $l$th dyadic piecewise linear approximation,
i.e., $h(l)_{Tj /2^k} = h_{Tj /2^l}$ and
$h(l)$ is linear on $[T(j-1) /2^l, Tj /2^l]$
for every $j~(1 \le j \le 2^l)$.
Then, it is easy to see that $h(l) \to h$ in ${\mathcal H}$
as $l \to \infty$.
It is also known that
${\bf h}(l) := {\bf L}(h(l))
  \to {\bf h}$ as $l \to \infty$
in $G\Omega_{\alpha} ([0,T], {\mathbb R}^d)$,
where ${\bf h}$ is the natural lift of $h$
(namely, its second level path is given by
${\bf h}^2_{s,t}
  = \int_s^t (h_u -h_s) \otimes h^{\prime}_u du$).
We write ${\bf h} =L(h)$.

Denote by $x^h$
the solution of ODE \eqref{RDE2} in the Riemann-Stieltjes
sense with $w$ being replaced by $h$
(the starting point is $m$ as before).
Then, ${\mathcal H} \ni h \mapsto x^h \in C_L([0,T], {\mathcal M})$ is continuous.
Then, by taking the limit of $x^{h(l)} = \Phi (k, {\bf h}(l))$,
we have
$x^{h} = \Phi (m, {\bf h})$.

Now we present a support theorem of
Stroock-Varadhan type.
In what follows, we denote by $\tilde{\Phi}$
the continuous map
${\bf w} \mapsto
  \bigl[ t \in [0,T] \mapsto
    \Phi (\bullet, {\bf w})_t \bigr]
$
given in Proposition \ref{pr.diffeo}.
When ${\bf W}$ is Brownian rough path,
denote by $\nu$ the law of the
$C([0,T], {\rm Diff}_{L}^{k} ({\mathcal M}) )$-valued random variable
$
  \tilde\Phi ({\bf W})
$.

\begin{proposition}\label{pr.support}
  Keep the same notation and assumptions
  as in Proposition \ref{pr.diffeo}.
  Then, the (topological) support of $\nu$
  is given by the closure of
  $
    \{  \tilde\Phi ({\bf h}) \colon  h \in {\mathcal H} \},
  $
  where the closure is taken with respect to
  the topology of $C([0,T], {\rm Diff}_{L}^{k} ({\mathcal M}) )$.
\end{proposition}

\begin{proof}
  First, recall that $\{ {\bf h} \colon  h \in {\mathcal H} \}$
  is dense in $G\Omega_{\alpha}([0, T], \R^{d})$
  and the support of the law of Brownian
  rough path is  $G\Omega_{\alpha}([0, T], \R^{d})$.
  (See \cite[Section 13.7]{FV} for example.)

  Since $\tilde\Phi$ is continuous
  and ${\bf L} ({\mathcal H})$ is dense,
  the image of $\tilde\Phi$ equals the closure of
  $\{  \tilde\Phi ({\bf h}) \colon  h \in {\mathcal H} \}$.
  Hence, the support of $\nu$ is included
  in the closure.
  We prove the reverse inclusion.
  Take any $\xi$ from the closure.
  Then,
  since any open neighborhood of $\xi$ contains
  at least one $\tilde\Phi ({\bf h})$,
  its inverse image by $\tilde\Phi$ is a non-empty
  open subset in $G\Omega_{\alpha}([0, T], \R^{d})$
  and is therefore of positive measure.
  Hence, $\xi$ belongs to the support of $\nu$.
\end{proof}

Next we present a large deviation principle of
Freidlin-Wentzell type.
For more information on large deviations,
see Dembo-Zeitouni \cite{DZ} for example.

For $\epsilon >0$, we consider the $\epsilon$-scaled
version of  RDE \eqref{RDE2}:
\begin{equation}\label{RDEscale}
  dx_{t}= \epsilon
  \sum_{i=1}^d V_i (x_{t})dw^i_{t} + V_0 (x_{t})dt.
\end{equation}
Its solution with initial value $m$
actually coincides with
$\Phi (m, \epsilon{\bf w})$, where we set
$\epsilon{\bf w} =(\epsilon{\bf w}^1, \epsilon^2{\bf w}^2)$.
When ${\bf W}$ is Brownian rough path,
denote by $\nu^{\epsilon}$ the law of
$
  \tilde\Phi (\epsilon{\bf W})
$.

\begin{proposition}\label{pr.LD}
  Keep the same notation and assumptions
  as in Proposition \ref{pr.diffeo}.
  Then,
  $\{ \nu^{\epsilon}\}_{\epsilon >0}$ satisfies a large deviation
  principle as $\epsilon \searrow 0$
  with a good rate function
  $J \colon C([0,T], {\rm Diff}_{L}^{k} ({\mathcal M}) )
    \to [0, \infty]$, where
  $J(\xi) := \inf \{ \|h\|^2_{{\mathcal H}}/2  \colon
    \mbox{ $h \in {\mathcal H}$ with $\xi = \tilde \Phi ({\bf h})$}
    \}$
  (it is  understood that $\inf \emptyset = \infty$).
\end{proposition}

\begin{proof}
  Recall that the law of $\epsilon {\bf W}$ satisfies
  a large deviation
  principle as $\epsilon \searrow 0$
  with a good rate function
  $I \colon
    G\Omega_{\alpha}([0, T], \R^{d}) \to [0, \infty]$,
  where
  $I({\bf w}) :=  \|h\|^2_{{\mathcal H}}/2$
  if ${\bf w} = {\bf L} (h)$ for some $h \in {\mathcal H}$
  and $I({\bf w}) :=\infty$ if no such $h \in {\mathcal H}$ exists.
  (See \cite[Section 13.6]{FV} for example.)
  From this fact, the continuity of $\tilde\Phi$,
  and the contraction principle,
  the claim of this proposition immediately follows.
  (For the contraction principle,
  see Dembo-Zeitouni \cite[Theorem 4.2.1]{DZ},
  in which the target set of a continuous map
  is only required to be a Hausdorff topological space.)
\end{proof}

\noindent
{\bf Acknowledgement:}~
The authors are grateful to Professors Hitoshi Moriyoshi
and Atsushi Atsuji for stimulating discussions
and helpful comments.
The first named author
is partially supported by JSPS KAKENHI Grant Number JP15K04922
and JP18F18314.



\bigskip
\begin{flushleft}
  \begin{tabular}{ll}
    Yuzuru \textsc{Inahama}
    \\
    Faculty of Mathematics,
    \\
    Kyushu University,
    \\
    744 Motooka, Nishi-ku, Fukuoka, 819-0395, JAPAN.
    \\
    Email: {\tt inahama@math.kyushu-u.ac.jp}
  \end{tabular}
\end{flushleft}

\begin{flushleft}
  \begin{tabular}{ll}
    Kiyotaka \textsc{Suzaki}
    \\
    Headquarters for Admissions and Education,
    \\
    Kumamoto University,
    \\
    Kurokami 2-40-1, Chuo-ku, Kumamoto, 860-8555, JAPAN.
    \\
    Email: {\tt k-suzaki@kumamoto-u.ac.jp}
  \end{tabular}
\end{flushleft}
\end{document}